\documentclass[a4paper,11pt]{amsart}

\usepackage{hyperref}
\hypersetup{
    colorlinks=true,     
    linkcolor= blue,     
    citecolor=blue,        
    filecolor=magenta,     
     urlcolor=cyan           
}

\usepackage{a4wide}
\usepackage{amsmath}
\usepackage{amsfonts}
\usepackage{amssymb}
\usepackage[all]{xy}
\usepackage{mathrsfs}
\usepackage{stmaryrd}
\usepackage{oldgerm}
\usepackage[OT2,T1]{fontenc}
\usepackage[latin1]{inputenc}
\usepackage{bbm}

\DeclareSymbolFont{cyrletters}{OT2}{wncyr}{m}{n}
\DeclareMathSymbol{\Sha}{\mathalpha}{cyrletters}{"58}

\def \ooverline #1#2#3%
{\mkern #1mu \overline{\mkern -#1mu #3 \mkern -#2mu }\mkern #2mu }

\DeclareMathOperator{\Sel}{Sel}

\DeclareMathOperator{\Ev}{Ev}
 
\DeclareMathOperator{\Gal}{Gal} 
     
\DeclareMathOperator{\Ad}{Ad}
  
\DeclareMathOperator{\f}{{f}}  \DeclareMathOperator{\T}{T}
\DeclareMathOperator{\SL}{\mathsf{SL}}
\DeclareMathOperator{\End}{\mathsf{End}} 
\DeclareMathOperator{\Hom}{\mathsf{Hom}}
 
\DeclareMathOperator{\Ker}{Ker} \DeclareMathOperator{\im}{Im} 
\DeclareMathOperator{\Coker}{Coker}  
  
\DeclareMathOperator{\Tor}{Tor}  \DeclareMathOperator{\SO}{SO} 
\DeclareMathOperator{\depth}{depth}\DeclareMathOperator{\Ass}{Ass}\renewcommand{\H}{\mathrm{H}} 
\DeclareMathOperator{\Fitt}{Fitt}
\DeclareMathOperator{\Frob}{Frob}
 
\DeclareMathOperator{\Irr}{Irr} 
\DeclareMathOperator{\Min}{Min}
\DeclareMathOperator{\ord}{ord}

\DeclareMathOperator{\CH}{\mathsf{char}}
  
\DeclareMathOperator{\alg}{alg} 

\DeclareMathOperator{\W}{W} \DeclareMathOperator{\A}{A}   \DeclareMathOperator{\Lie}{Lie}
\DeclareMathOperator{\Res}{Res}

\DeclareMathOperator{\Gm}{\mathbf{G_m}} 
\DeclareMathOperator{\GL}{\mathsf{GL}}

\DeclareMathOperator{\Symm}{\mathsf{Symm}} 
\DeclareMathOperator{\Ann}{\mathsf{Ann}}
\DeclareMathOperator{\pd}{\mathsf{pd}}
\DeclareMathOperator{\Card}{\mathsf{Card}}

\DeclareMathOperator{\Tr}{Tr}

\newcommand{\cO}{\mathcal{O}}
\newcommand{\bC}{\mathbb{C}}

\newcommand{\bR}{\mathbb{R}}

\newcommand{\bQ}{\mathbb{Q}}
\newcommand{\cK}{\mathcal{K}}

\newcommand{\ff}{\mathbf{f}}
\newcommand{\fg}{\mathbf{g}}

\newcommand{\cV}{\mathcal{V}}
\newcommand{\bA}{\mathbf{A}}

\def\im{\mathop{\rm  }}
\def\ker{\mathop{\rm Ker}}

\def\p{{\frak p}}

\def\h{{\frak h}}
\def\c{{\frak c}}

\def\n{{\frak n}}

\def\l{{\frak l}}

\def\C{{\Bbb C}}
\def\m{{\frak m}}
\def\s{{\frak s}}
\def\c{{\frak c}}

\def\lb{\lambda}

\def\N{{\Bbb N}}
\def\Z{{\Bbb Z}}
\def\Zp{{{\Bbb Z}_p}}
\def\Q{{\Bbb Q}}
\def\e{{\bf e}}
\def\R{{\Bbb R}}
\def\C{{\Bbb C}}
\def\F{{{\Bbb F}}}

\def\adots{\mathinner{\mkern2mu\raise1pt\hbox{.}\mkern3mu\raise4pt\hbox{.}\mkern1mu\raise7pt\hbox{.}}}

\makeatletter
\@addtoreset{equation}{section}

\makeatother

\newtheorem{thm}[equation]{Theorem}
\newtheorem{cor}[equation]{Corollary}
\newtheorem{pro}[equation]{Proposition}
\newtheorem{lem}[equation]{Lemma}

\newtheorem{de}[equation]{Definition}
\newtheorem{rem}[equation]{Remark}
\newtheorem{conj}[equation]{Conjecture}

\title{Integral periods relations and congruences}
\author{J. Tilouine and E. Urban}
\thanks{  The first author is partially supported by ANR grants PerCoLaTor ANR-14-CE25 and CoLoss ANR-19- and by the NSF grant
DMS 1464106. 
The second author is partially funded by the grant DMS 1407239 from the National Science Foundation. }
\begin{document}

\begin{abstract}

Under relatively mild  and natural conditions, we establish an integral period relations for the (real or imaginary) quadratic base change
of an elliptic cusp form. This answers a conjecture of Hida regarding the {\it congruence ideal} controlling the congruences
between this base change and other eigenforms which are not base change. As a corollary, we establish the Bloch-Kato 
conjecture for adjoint modular Galois representations twisted by an even quadratic character. In the odd case, we formulate
a conjecture linking the degree two topological period attached to the base change Bianchi modular form, the cotangent complex
of the corresponding Hecke algebra and the archimedean regulator attached to some Beilinson-Flach element.

\end{abstract}

\maketitle

\section{Introduction}

Describing the conjectural Deligne periods (or Beilinson regulators) at a special value $s=s_0$ of a motive $ M=T(N_1,\ldots ,N_r)$ 
obtained by plethysm (a combination of tensor products, direct sums or duals) of other motives $N_i$ 
in terms of the conjectural Deligne periods of the $N_i$'s (at the same special value $s_0$) is a rather straightforward
linear algebra computation. On the other hand, when this motive is associated to an automorphic representation 
$\Pi=\varphi_\ast(\pi)$ obtained
by Langlands functoriality for a morphism $\varphi\colon{}^LH\to{}^LG$ of $L$-groups, establishing the period relation 
up to an algebraic number between the automorphic periods obtained from the $\Pi$-isotypic part of the
 cohomology of the Shimura variety, resp. of the locally 
symmetric domain, for $G$  and the automorphic periods coming from the $\pi$-isotypic part of the
 cohomology of the Shimura variety, resp. of the locally 
symmetric domain, for $H$ is much more difficult. In fact, if these groups admit Shimura varieties, 
it is closely related to the Hodge and Tate conjectures for these motives. 
Nevertheless, period relations
without assuming Hodge and Tate conjectures have been obtained in many cases from integral representations 
of L-functions together with certain non-vanishing results or assumptions. A typical structure of the proof would be

1) On one hand, given a rational representation $r\colon {}^LG\to \GL_m$, one finds integral representation formulas  for 
$L(\Pi,r,s)$, resp. for $L(\pi,r_i,s)$, where 
\begin{equation}\label{Lgroup-rep-fact}
r\circ\varphi=r_1\oplus\ldots\oplus r_t
\end{equation}

2) One proves the non-vanishing of the special values $L(\Pi,r,s_0)$, resp. for $L(\pi,r_i,s_0)$ (this may require twists),

3) Then one proves that these integral representations have rational structures (for $s=s_0$) 
(using algebraic topology) and one deduces from this rationality theorems for
the normalized (twisted) $L$-value $L^\ast(\Pi,r,s_0)=\frac{L(\Pi,s_0)}{\Omega_\Pi(r)}$,
resp. $L^\ast(\pi,r_i,s_0)=\frac{L(\pi,r_i,s_0)}{\Omega_\pi(r_i)}$,

4) Using the factorization $L(\Pi,r,s)=\prod_i L(\pi,r_i,s)$ and taking the ratio, one relates up to a non-zero algebraic number
$\Omega_\Pi(r)$ to $\prod_i\Omega_\pi(r_i)$.

Many results have been obtained using this kind of approach. For example, the reader can consult the works of G. Shimura \cite{Shimura}, D. Blasius \cite{Blasius}, 
M. Harris \cite{Harris} and more recently by H. Grobner, M. Harris and J. Lin \cite{GHL} just to quote a few.

On the other hand, establishing the same period relations up to a $p$-adic unit with similar arguments 
would require to
use the non-vanishing modulo $p$ of normalized twisted $L$-values. Since very few of those are known, 
it turns out that proving integral period
relations is usually very difficult and very few general examples are known (see for instance \cite{GV00}, \cite{PV16}). 

Actually, the integral period relations can be deduced from a form of the Bloch-Kato conjectures for the motives associated to $\pi$ and
to $\Pi=\varphi_\ast(\pi)$\footnote{Recall that we denote by $\varphi_*$ the map giving the conjectural Langlands functoriality induced by $\varphi$. In theory, we should in fact speak about transfer of packets but we omit this technicality which does not matter for this introduction.} when the arithmetic $L$-values defined previously can be related to the size of the $p$-adic Selmer groups attached to
the Galois representations determined by the representations of the $L$-groups at play. The factorization \eqref{Lgroup-rep-fact} induces then a 
direct decomposition of Selmer groups and therefore a factorization of the arithmetic $L$-values and therefore of the periods up to a $p$-adic unit.
More precisely, for $(\Pi,r)$ and $(\pi,r_i)$, one needs to prove an analogue of the Bloch-Kato conjecture expressing the special values at $s_0$ 
of these automorphic $L$ functions normalized by automorphic periods as (respectively) the order of the $\Sha^1$ of $r_\ast T_\Pi(s_0)\otimes\Q_p/\Z_p$, where $T_\Pi(s_0)$ 
is a $p$-adic lattice in the $p$-adic realization of the motive $M_\Pi(s_0)$ and $\Sha^1(r_\ast T_\Pi(s_0)\otimes\Q_p/\Z_p)$ is the quotient 
of the Selmer group of $r_\ast T_\Pi(s_0)\otimes\Q_p/\Z_p$ by its $p$-divisible part, resp. the order of the $\Sha^1$ of $r_{i,\ast} T_\pi(s_0)\otimes\Q_p/\Z_p$, with similar notations. 
Then one compares these formulas for $(\Pi,r)$ and $(\pi,r_i)$'s where $\varphi_\ast(\pi)=\Pi$. 
This yields an integral period relation because the special $L$-value of $(\Pi,r)$, resp. the corresponding Selmer group and $\Sha$-group, 
decomposes as a product, resp. a direct sum, over $i$'s.
\medskip

The purpose of the present work is to contribute to the case where the Langlands functoriality is the base change
$$\varphi\colon {}^L\Res_{F_1/\Q}\GL_2\to{}^L\Res_{F_2/\Q}\GL_2$$
for cuspidal representations,
where $F_2/F_1$ is an abelian extension of number fields, mostly when $F_1=\Q$ and $F_2$ is quadratic, although in certain sections, 
$F_2$ is abelian totally real.
However none of the methods alluded above can be completed
fully and give only partial results. Our strategy is to show we can combine them through the study of congruences  in order to obtain the integral period relation for quadratic base change. 

Let us introduce some notations.
We denote by $\Ad$ the degree $3$ representation given by the adjoint action of $\GL_2$ on $\s\l_2=\Lie(\SL_2)$.
and let $r=\Ad$.
Let $f\in S_k(\Gamma_0(N))$ be a primitive cuspform of even weight $k\geq 2$ and level $N$.
Let $h_k(N,\Z)\subset \End\,S_k(\Gamma_0(N))$ be the ring generated by Hecke operators outside $N$. 
Let $K_f$ be the field generated by the Hecke eigenvalues of $f$.
Let $\alpha$ be a Dirichlet character and $\nu=0,1$ such that $\alpha(-1)=-(-1)^\nu$.
We define, following \cite[Section 2.5]{H99} the gamma factor of the motive $\Ad(f)\otimes\alpha$ as
$$\Gamma(\Ad(f)\otimes\alpha,s)=\Gamma_\C(s+k-1)\Gamma_\R(s+\nu)$$
where $\Gamma_\C(s)=2\cdot(2\pi)^{-s}\Gamma(s)$ and $\Gamma_\R(s)=\pi^{-s/2}\Gamma({s\over 2})$.
By self-duality of $\Ad(f)$, the functional equation \cite{GJ78} for the complete $L$ function 
$$\Lambda(\Ad(f)\otimes\alpha,s)=\Gamma(\Ad(f)\otimes\alpha,s)L(\Ad(f)\otimes\alpha,s)$$
relates $s$ and $1-s$.
The special value at $s=1$ is critical if and only if $\alpha$ is even (that is, $\nu=1$).
However, there are precise rationality and even integrality results. 
If $\alpha$ is an even Dirichlet character of conductor prime to $N$, 
let $G(\alpha)$ be the Gauss sum associated to $\alpha$ and 
$K_f(\alpha)$ be the field generated over $K_f$ by the values of $\alpha$. we have 
$$  L^\ast(\Ad f \otimes \alpha):= G(\overline{\alpha})^2{\Gamma(\Ad f\otimes\alpha ,1)L(\Ad f\otimes\alpha ,1)\over \Omega_{f}^{+}\Omega_{f}^{-}}\in K_f(\alpha)$$
(see \cite{St80} and \cite{H90}). See Section \ref{sectManinCongruence} 
below for the precise definition of the Manin periods $\Omega_{f}^{+}$ and $\Omega_{f}^{-}$.
Concerning integrality results, we will assume that $\alpha$ is quadratic but not necessarily even
(although when $\alpha$ is even, it does not need to be quadratic to prove integrality of these values). 
These integrality statements were established by 
Hida \cite{H99}  for $\alpha$ even or odd
(the periods are different in the odd case, as expected, since it is a non-critical value).

Let now $F$ be the quadratic field associated to $\alpha$.
 Hida  formulated  a Conjecture \cite[Conjecture 5.1]{H99} relating the value $ L^\ast(\Ad f \otimes \alpha)$
to the cardinality of a non base change congruence module of $C_1$-type for the field $F$.
We refer to the first section of this paper for the formal technical definition of these congruence modules\footnote{We refer to Hida's original paper \cite{H86} which explains the motivation for this terminology.} of type $C_0$ and $C_1$, and will just say here that these commutative algebra invariants measure the existence and the amount of congruences between the base change to $F$ of our elliptic eigenform $f$ with  Hilbert or Bianchi modular forms for $F$ which do not come from base change up to twist.
In this paper, we prove this conjecture as stated by Hida for an even quadratic character.
In the odd case, we prove a similar statement but one needs to replace the congruence module of $C_1$-type
by the one of $C_0$-type. Their orders may not be equal in that case as the Hecke algebra 
is not complete intersection in general.

These results are proved at the same time as an integral periods relation for the quadratic base change map 
$\varphi\colon {}^L\GL_2{}_{/\Q}\to{}^L\Res_{F/\Q}\GL_2$. 
We prove this periods relation by proving two divisibilities between periods and certain
 L-values and non base-change congruence ideals by using two complementary tools. 
 The first one is a theorem of Cornut and Vatsal \cite{C02} and \cite{Va02} (and Chida-Hsieh \cite{CH16}) 
stating the non-vanishing modulo 
$p$ of twisted standard $L$ values. Their theorem requires an assumption which we will need to assume

\medskip
(CV) There exists a prime $\ell$ such that $\overline{\rho}_f\vert_{\Gamma_{\Q_\ell}}$ is indecomposable, If there is a prime $q|N$ such that $q^2|N$ and $q\equiv -1\pmod p$, then we assume that $\overline{\rho}_f$ restricted to the Inertia subgroup at $q$ is irreducible.

\medskip
The second one is Hida's integral linear form on $\H^2$ of the modular variety for $\Res_{F/\Q}\GL_2$ 
(either a Hilbert surface or a Bianchi hyperbolic threefold) whose kernel is generated by the Hecke eigenclasses on $\H^2$ 
associated to cusp forms on $\Res_{F/\Q}\GL_2$ which do not come (up to twist)  from base change. 
If the theorems of Cornut-Vatsal \cite{C02} and \cite{Va02} (and Chida-Hsieh \cite{CH16}) was not requiring $f$ to have a trivial nebentypus, we could actually 
prove Hida's Conjectures for an newform $f\in S_k(\Gamma_1(N))$ (as Hida's integral linear form on $\H^2$ does allow forms with non-trivial Nebentypus).
However, since these results are not yet available we restrict ourselves to the trivial Nebentypus case in this paper.  

\medskip
In order to state our main results, let us be more precise.  Throughout this work $p$ will be an odd prime.
We fix once for all an embedding $\overline{\Q}\to\overline{\Q}_p$ and an isomorphism $\overline{\Q}_p\cong \C$.
Let $\cO$ be the valuation ring of a  $p$-adic field 
containing the eigenvalues of $f$. It can be viewed as a subring of $\C$ via a fixed identification $\bar\Q_p\cong\C$.
We fix a uniformizing parameter $\varpi\in\cO$.
Let $\m$ be the maximal ideal of the Hecke algebra $h_k(N;\cO)=h_k(N;\Z)\otimes \cO$ associated to the reduction of $f$
modulo $\varpi$. We denote by $\T$ the localization of this Hecke algebra at $\m$. 
For any number field $E$, let $\Gamma_E=\Gal(\overline{\Q}/E)$.
We will assume throughout that the reduction $\overline{\rho}_f$ of the $p$-adic Galois representation $\rho_f\colon\Gamma_\Q\to\GL_2(\cO)$ 
associated to $f$ is irreducible. 
In other words, $f$ is non-Eisenstein at $p$.

\medskip
Let us assume first that $F$ is real quadratic or in other words that  $\alpha$ is even. We denote $D$ its discriminant.
%
Let $f_F$ be the normalized base change of $f$ to $F$ 
and $\m_F$ the maximal ideal of the cuspidal Hilbert Hecke algebra $h_{k,k}^F(N;\cO)$ associated to $f_F$.
We assume throughout the paper that $\m_F$ is non-Eisenstein as well. It would follow for instance from the stronger assumption 
about $(f,p)$ that $\im\,\overline{\rho}_f$ contains $\SL_2(\F_p)$.
Let
$\T_F$ be the localization of $h_{k,k}^F(N;\cO)$ at $\m_F$.
For every character $\epsilon$ of $\{\pm 1\}^2$ (which we will identify as a pair of signs)
Hida \cite{H99}, (see also Section \ref{periodsF} below) one defines
periods $\Omega_{f_F}^\epsilon\in\C^\times/\cO^\times$ analogous to the Manin periods. 
Then, our main results (see section  \ref{sectionRealQuadratic}) are the following.

\begin{thm}\label{thmA} We assume that  $p>k-2$ and $p$ is prime to $6ND\cdot \#(\cO_F/N\cO_F)^\times$, that $\m_F$ is not Eisenstein and that (CV) holds. Then,  
 for $\epsilon= (+,-)$ or $(-,+)$,  we have the period relation :
$$ \Omega_{f_F}^{\epsilon}\sim  \Omega_{f}^{+}\Omega_{f}^{-}.$$
where, for $A,B\in\C^\times$, we write $A\sim B$ if $A=B\cdot u$ for some $u\in\cO^\times$.
\end{thm}

Following Hida, we define (in sections \ref{conggensect} and  \ref{sectManinCongruence}) congruence 
ideals  $\eta_f$, $\eta_{f_F}$ and $\eta_{f}^{\sharp}$ controlling the congruences between $f$
and other cusp forms, resp. its base changes to other Hilbert cusp forms, resp. its base change to non base change Hilbert cusp forms.
We make some assumptions on the Galois representation $\rho_f$:

\begin{itemize}
\item $\rho_f\vert_{\Gamma_F}$ is $N$-minimal: for any prime $v$ in $F$ dividing $N$, and for $I_v$ an inertia subgroup at $v$, 
either $\ord_v(N)=1$ and $\overline{\rho}_f(I_v)\neq 1$,
or $\ord_v(N)>1$ and the reduction map induces an isomorphism $\rho_f(I_v)\cong \overline{\rho}_f(I_v)$ 
and the latter group acts irreducibly on $k^2$,
\item $p>k$,
\item $\overline\rho_f|_{G_{F'}}$ is irreducible (see Section \ref{sectManinCongruence} below), 
\end{itemize}

\begin{rem} Note that if $\fg$ is a Hilbert cusp eigenform of level $N$ such that $\overline{\rho}_{\fg}$ is $N$-minimal,
it follows from Carayol's theorem that $\fg$ has conductor $N$, hence is new.
 Hence, under the $N$-minimality assumption, all cuspforms occurring in $\T_F$ are newforms.
\end{rem}
 
Let $R$, resp. $R_F$, be the universal deformation ring for deformations of $\overline{\rho}_f$, resp. $\overline{\rho}_f\vert_{\Gamma_F)}$,
which are $N$-minimal and Fontaine-Laffaille of weights $0$ and $k-1$ at $p$, resp. above $p$. 
We have $R=\T$ for $f$ by \cite{Wi95}, 
and similarly,  $R_F=\T_F$ for $f_F$ by \cite{Fu06} (where $N$-minimal\footnote{This condition is recalled as (Min) in the text before Theorem \ref{classical}.} is called $N$-finite).
We can therefore rewrite the congruence ideals in terms of $\cO$-Fitting ideals of (Bloch-Kato) Selmer groups  $(\Sel(L,\Ad\,\rho_{f})$, $L=\Q$ or $F$,
of the adjoint Galois representation $\Ad\rho_f$. 
 Let $F'=F(\zeta_p)$ with $\zeta_p$ a primitive $p$-th root of unity.
For any finite $\cO$-module $M$, let $\CH_\cO(M)=\Fitt_\cO(M)$
be its Fitting ideal. 

\begin{thm}\label{thmB} We keep the same assumptions as in Theorem \ref{thmA} and further suppose that $p>k$, $\overline\rho_f|_{G_{F'}}$ 
is irreducible and $N$-minimal, then we have:
 $$\CH_\cO(\Sel_\Q( \Ad\,\rho_{f}\otimes\alpha))= (L^\ast(\Ad f \otimes \alpha))$$
 as ideals of $\cO$.
\end{thm}
This is Corollary \ref{BKreal} of the text.

\begin{rem} Since we assume that $\rho_f$ is $N$-minimal, we see that if there exists a prime $\ell$ dividing $N$ such that $\ord_\ell(N)=1$, the
assumption (CV) is satisfied. Hence, if such a prime $\ell$ exists, one can remove the condition (CV) in the statement of Theorem \ref{thmB}.
This is true in particular if $N$ is squarefree.
\end{rem}

For an elliptic curve, this implies the following corollary.
\begin{cor} Let  $E/\Q$ be a semistable elliptic curve of conductor $N$ and $p>3$ a prime for which $E$ has no subgroup of order $p$ defined
over $F'$, then if the Galois representation $E[p]$ has conductor $N$  and $p$ is prime to $ND\# (O_F/NO_F)^\times$, we have
$$\CH_\cO (\Sel_\Q(\Ad(E)\otimes \alpha))=(L^\ast(\Ad(E)\otimes\alpha,1))$$
where $$L^\ast(\Ad(E)\otimes\alpha):={\Gamma(\Ad(E)\otimes\alpha,1)L^\ast(\Ad(E)\otimes\alpha,1)\over \Omega_E^+\Omega_E^-}$$
and $\Omega_E^+,\Omega_E^-$ are the N\'eron differentials periods associated to $E$.
\end{cor}

\medskip
Let us consider now the case  $F$  imaginary quadratic with discriminant $-D$; so now $\alpha$ is odd. We obtain results 
of a different nature as 
 the value of $L(\Ad(f)\otimes\alpha,s)$ at $s=1$ is no longer  critical. 
We introduce similarly the normalized base change $f_F$ of $f$ to $F$. This is a Bianchi modular form
of weight $(k,k)$ and level $N.\cO_F$ (see for instance \cite[Section 2]{Lan80} and \cite{Ku80}). One can associate two differential forms on the Bianchi variety
of respective degree one and two. This allows to define two periods $u_1(f_F)$ and $u_2(f_F)$ (see \eqref{periods} in Section \ref{periods-imagquad} for their definitions) 
as in the second author's original work 
\cite{Ur95}.

let $h_{k,k}^F(N,\cO)$ be 
the Hecke algebra acting faithfully on the second cohomology of the Bianchi variety of level $N$ and
$\T_F$ be its localization at the maximal ideal $\m_F$ associated to $f_F$.
Let $\eta_f$ be the congruence ideal\footnote{This is the Fitting ideal of the congruence module of $f$. See section \ref{sectManinCongruence} where  the definition of this classical invariant is recalled.} for $f$.  We define (see Section \ref{sectHeckecong})
Hecke and cohomological congruence ideals  $\eta_{f_F}^{i}$ and 
$\eta_{f}^{i,\sharp}$ for $i\in\{1,2\}$ which measure respectively congruences of the base change $f_F$ with all other Bianchi modular forms of same level and weight and congruences with those which does noe come from Base change up to twist. The indexes $i$ here means that in each case the integral structure considered is the one coming from the degree i integral cohomology of the Bianchi hyperbolic threefold. Our main result in the imaginary quadratic field case is the following Theorem.

\medskip
\begin{thm} \label{thmC} Assume that $(N,D)=1$, that (CV) holds and that the number of prime factors of $N$
which remain inert in $F$ is odd.  Let $p>k$ be a prime not dividing $D\cdot N\cdot \# (O_F/NO_F)^\times$ and suppose that $f_F$ is non-Eisenstein 
and $N$-minimal for this prime.
Then

\medskip
(i) We have the period relation (up to a $p$-adic unit)
$$u_1(f_F) \sim \Omega_f^+\Omega_f^-,$$

(ii) We have the equality (up to a $p$-adic unit)
$${\Gamma(\Ad(f)\otimes\alpha,1) L(\Ad(f)\otimes\alpha,1)\over u_2(f_F)}\sim \eta_{f}^{1,\sharp}$$

(iii) We have an equality of ideals in $\cO$:
$$\eta_{f_F}^{1}=\eta_f\cdot \eta_{f}^{1,\sharp}
$$

\end{thm} 

Note that the last statement is not a formal consequence of Hida's congruence ideal formalism which does not apply {\it a priori} 
since the Hecke algebra
$\T_F$ is not Gorenstein in general.

\medskip
By using the works of Caraiani-Scholze \cite{CS19}, Newton-Thorne \cite{NT16}, Caraiani and {\it al} \cite{CGHJMRS} and \cite{ACC+}, 
we know that there is a deformation of $\overline{\rho}_f\vert_{\Gamma_F)}$ to $\T_F$, which is $N$-minimal. Unfortunately, it is only conjecturally Fontaine-Laffaille of weights $0$ and $k-1$, 
because the ground field $F$ is too small for the arguments in \cite[Theorem 4.5.1]{ACC+}.
For any Taylor-Wiles squarefree ideal $Q$ of $\cO_F$ prime to $Np$, let $\T_{F,Q}$ be the local Hecke algebra of weight $(k,k)$ and level $U_0(N)\cap U_1(Q)$ defined in 
\cite[Conjecture A]{CaGe18}.  
We need to assume, as in 
\cite[Conjecture A]{CaGe18}, that  

\medskip
$(FL)$ for any Taylor-Wiles square free ideal $Q$ prime to $Np$, there exists a deformation of $\overline{\rho}_f\vert_{\Gamma_F)}$ to $\T_{F,Q}$ 
which is $NQ$-ramified, $N$-minimal and Fontaine-Laffaille of weights $0$ and $k-1$ at all places above $p$ in $F$, and Taylor-Wiles at $Q$, 

\medskip
Then, by \cite[Theorem 5.11 and Remark 5.13]{CaGe18}, the universal homomorphism $\phi_F\colon R_F\to \T_F$ is an isomorphism. 
this is a key ingredient for the following theorem. We denote by $\lb_F\colon \T_F\to \cO$ the Hecke eigensystem associated to $f_F$.

\begin{thm} \label{ThOPTintro} We keep the same hypotheses as in the previous Theorem and we assume (FL). 
Then, we have
$$\CH_\cO \Sel_\Q(\Ad(\rho_f\otimes\alpha)) =({\Gamma(\Ad(f)\otimes\alpha,1)L(\Ad(f)\otimes\alpha,1)\over u_2(f_F)})\cdot  \CH_\cO(\H_1(L_{T_F/\cO}\otimes_{\lb_{f_F}}\cO))$$
where $L_{T_F/\cO}$ stands for the cotangent complex of $T_F$ over $\cO$.
\end{thm}

The characteristic ideal of the degree one homology of the cotangent complex measures the defect of $T_F$ of being complete intersection. This defect
is always non trivial since it was observed in \cite{CaGe18} that $\T_F$ is not complete intersection in this context.

A formulation of the Bloch-Kato conjecture in this setting leads us to the
following conjecture that can be seen as an integral version of a conjecture of Prasanna-Venkatesh.
Let $R_{f,\alpha}\in\R^\times$ be the archimedean regulator of a 
Beilinson-Flach element attached to $f$ and $\alpha$. It is defined in Section \ref{subsectBKimquad},
just before Conjecture \ref{BK-non-crit}. It is well-defined up to an element of $\cO^\times$.
As mentioned in Section \ref{subsectBKimquad}, its relation to the special twisted adjoint $L$ value at $1$ requires 
a stronger Fontaine-Laffaille constraint: $p>2k-1$.
Thus we include it in the following conjecture.

\begin{conj} We keep the hypothesis of Theorem \ref{ThOPTintro} and we assume $p>2k-1$, then the following holds:
$$  u_2(f_F)\sim  {\Omega_f^+\Omega_f^-\cdot R_{\f,\alpha} \cdot
\CH \H_1(L_{T_F/\cO}\otimes_{\lb_{f_F}}\cO)}$$
\end{conj}

It results from Theorem \ref{ThOPTintro} and Lemma \ref{Sha-Selmer} that the above conjecture is equivalent to the Bloch-Kato Conjecture, 
stated as Conjecture \ref{BK-non-crit} below.

\medskip
The paper is organized as follows. In section 2, we gather some abstract results and definitions on congruence ideals and in particular
a general factorization of the congruence ideal in the base change setting. In Section 3, we recall the various definitions of periods and
congruences ideals for elliptic cusp forms and the main result about them. The section 4 is devoted to our result for the base change
to totally real fields. In particular, we formulate a general conjecture of an integral period relation for the base change and establish the result in the quadratic case.
In passing, we extend some facts about freeness of the cohomology that are not available in the literature. In section 5, we deal with the 
base change to an imaginary quadratic field and discuss our result in the the perspective of the Beilinson-Bloch-Kato conjecture and its link
with an integral version of the Prasanna-Venkatesh conjecture.

\medskip
\noindent
{\it Acknowledgment.} We would like to thank  M. Dimitrov and H. Hida for useful discussions and their interest in our work and for their
 comments on an earlier version of this paper. We are also happy to thank the referee for a careful reading of our paper, and in particular one question that  led us to write the remark 5.13.
 This project was started when the second author was a {\it Professeur invit\'e} at the University Paris 13 and he 
would like to thank this institution for its  financial support. 

\medskip

\tableofcontents
\newpage
\section{The formalism of the congruence modules}\label{conggensect}
The goal of this section is to recall some definitions and establish some  properties of congruence modules 
attached to characters of finite commutative algebras over a discrete valuation ring.
Some of these facts are well-known to the experts but are difficult to find in the literature except for the reference
\cite{H86}.
In this section $\cO$ is a discrete valuation ring of characteristic zero and $K$ is its field of fractions. 
For any finite $\cO$-module $F$, we will denote by $\Fitt_O(F)$ its Fitting ideal in $\cO$. For any finitely generated $\cO$-module $M$, we write
$M_K:=M\otimes_\cO K$.

\subsection{The congruence module $C_0$}\label{C0}
In this subsection, $T$ is a reduced, finite and flat local $\cO$-algebra or, equivalently, $T$ is finite flat over $\cO$ and $T_K$ is semi-simple. 
Now assume we have two semi-simple $K$-algebras $A_K$ and $B_K$ such that we have a factorization
\begin{equation}
\label{factorization}
T_K\cong A_K\times B_K
\end{equation}
We then will denote by $A$ (resp. by $B$) the image of $T$ into $A_K$  (resp. $B_K$) by the first (resp. the second) projection. 
For $C=A$ or $B$, let us call $\lb_C$ the  projection of $T$ onto $C$ and by $e_{\lb_C}\in T_K$ the corresponding idempotent of $T_K$.
For any finitely generated $T$-module $M$ which is flat over $\cO$, and for $C=A$ or $B$, we put
$$M_{\lb_C}:=  M\cap e_{\lb_C}\cdot M_K\mbox{ and } M^{\lb_C}= e_{\lb_C}\cdot M$$
When $\lb_C$ is clear from the context, we will write $M_C$ and $M^C$ in place of  $M_{\lb_C}$ and $M^{\lb_C}$. 
We put
\begin{eqnarray*}
C_0^{\lb_A}(M):= M^A/M_A
\end{eqnarray*}
This module is called a congruence module as it measures the congruences between the algebra homomorphisms in $A$ and $B$ that occur in $M$
as it is straightforward to see that the natural map $M\rightarrow M^A$, induces the isomorphisms:
$$\ { M\over M_A\oplus M_B}\cong M^A/M_A\cong {M^A\oplus M^B \over  M}$$
and in particular we have $C_0^{\lb_A}(M)\cong C_0^{\lb_B}(M)$.

\medskip
We now fix an $\cO$-algebra homomorphism;
$$\lb\colon T\rightarrow \cO$$
We denote by $T^c_K:= (1-e_{\lb})T_K$ and by $T^c$ the image of $T$ into $T^c_K$, so that we have an isomorphism $T_K\cong K\times T^c_K$ like in \eqref{factorization}.
Note that $\lambda(e_\lambda)=1$.
For any finite module $T$-module $M$ which is flat over $\cO$, we therefore have in particular
$$C^\lb_0(M):=M^\lb/M_{\lb}\cong M^{T^c}/M_{T^c}$$
and we put
$$\eta_\lb(M):= \Fitt_\cO(C_0^\lb(M))$$
In the special case $M=T$, we have $$\eta_\lb:=\eta_\lb(T)=\lb (\Ker (T\rightarrow T^c))=\lb(\Ann_T(\Ker(\lb)).$$
Note that $\eta_\lambda$ can also be viewed as an ideal of $T$, namely $T\cap e_\lambda\cdot T_K$. Indeed it is the largest ideal of $T^\lambda=\cO$ which is contained in $T$; in other words, it is the exact denominator 
in $\cO$ of $e_\lb$: if $x\in \cO$ and $x\cdot e_\lb\in T$, we see by applying $\lambda$
that $x\in \eta_\lambda$. Conversely, $\eta_\lambda\cdot e_\lb\subset T$ because $\eta_\lambda\subset T$ and we have $\eta_\lambda\cdot (1-e_\lb)=0$ in $T$.
Moreover, there is a canonical surjective map
$C_0^\lb(T)\otimes_T M\twoheadrightarrow C_0^\lb(M)$ and we can rewrite the left hand side as $C_0^\lb(T)\otimes_\cO M^\lambda$.
So if $M^\lb_K$ is of rank $1$, we have
\begin{eqnarray}\label{div-C0}
 C_0^\lb(M)=\cO/\eta_\lb(M)\quad \mbox{\rm and}\, \eta_\lb(M)\supset \eta_\lb
\end{eqnarray}
Indeed, by assumption, there exists $e\in M^\lb$ such that $M^\lambda=\cO\cdot e$. So there is a surjection of finite monogeneous $\cO$-modules
$C_0^\lb(T)\to C_0^\lb(M)$.
Or more concretely, we can write $e_\lb\cdot M=M^\lambda=\cO\cdot e$ 
and $M_\lambda=M\cap K\cdot e=\eta_\lb(M)\cdot e$. On the other hand, since $e\in\e_\lambda\cdot M$, we have 
$\eta_\lambda\cdot e\subset \eta_\lb\cdot e_\lb\cdot M\subset M\cap K\cdot e=M_\lambda$, hence 
$\eta_\lambda\cdot e\subset\eta_\lb(M)\cdot e$ as desired.

\medskip
Let $M^1$ and $M^2$ be two $T$-modules, flat over $\cO$ and such that there exists a perfect pairing
$$[\cdot,\cdot] \colon M^1\times M^2\rightarrow \cO$$
which is  $T$-bilinear: $[t\cdot m^1,m^2]= (m^1,t\cdot m^2]$ for any $t\in T$, and $m^i\in M^i$, $i=1,2$.
Then we have the following.

\begin{pro} \label{Pontryaginprop} With $A$ as above, the perfect pairing $[\cdot,\cdot] \colon M^1\times M^2\to \cO$ induces a perfect 
$A$-bilinear pairing 
$$[\cdot,\cdot]_A\colon C_0^A(M^1)\times C_0^A(M^2)\to K/ \cO
$$
in particular, we have $\Fitt_\cO(C_0^\lb(M_1))=\Fitt_\cO(C_0^\lb(M_2))$. Moreover, if $\lb$ is as above and if  the modules $M^1_\lb$, $M^2_\lb$ are rank $1$ 
over $\cO$ with respective basis $\delta_1$, $\delta_2$, then 
$$C^\lb_0(M^1)\cong C^\lb_0(M^2)\cong\cO/([\delta_1,\delta_2]). $$

\end{pro}

\begin{proof}
 For any $\cO$-submodule $L_1\subset M^1$, resp. $L_2\subset M^2$, let $L_1^\ast=\{x\in M^2_K; [L_1,x]\subset \cO\}$, resp.
$L_2^\ast=\{x\in M^1_K; [x,L_2]\subset \cO\}$ (note that these $\cO$-modules may contain $K$-vector subspaces). If we have two submodules $L_1\subset \Lambda_1\subset M^1$, we see immediately
that the pairing $[\cdot,\cdot]$ 
induces a perfect pairing
$$ \Lambda_1/L_1\times L_1^\ast/\Lambda_1^\ast\to K/\cO
$$
It is therefore enough to prove that $(a)\, M^{1,\ast}_{A}=M^{2,A}+e_{B} M^2_K$ and $(b)\, M^{1,A,\ast}=M^2_A+e_{B}M^2_K$. 
Let us prove $(a)$. For the inclusion $M^{2,A}+e_{B}M^2_K\subset M^{1,\ast}_{A}$, note first that $e_B M^2_K$ is orthogonal to
$e_A M^1_K$ since $e_A e_B=0$; 
for $m_1\in M^1_{A}$ and $m_2\in M^2$, we have $[m_1,e_{A}m_2]=[e_{A}m_1,m_2]$ and since $m^1=e_A m^1$, we conclude
$[m_1,e_{A}m_2]=[m_1,m_2]\in\cO$;
Conversely, let $e_{A} m_2\in M^2_K$ and assume that $[M^1_{A},m_2]\subset \cO$. Since $M^1_{A}$ is an $\cO$-direct factor in $M^1$, 
the linear form $[\cdot,m_2]$ extends to $M^1$. It can be therefore represented as $[.,t]$ for some $t\in M^2$. 
This shows that $m_2-t\in e_B M^2_K$ by orthogonality. Decomposing $t=e_A t+e_B t$, we see that $m_2\in e_A M^2+e_B M_K^2$ as desired.
The proof of $(b)$ is similar.

We now prove the last assertion.  We consider the map from $M^{1,\lb}\rightarrow \cO$ defined by $m\mapsto [m,\delta_2]$. 
From the discussion above, it is surjective since $M^{1,\lb}$ identifies to the $\cO$-dual of  $M^2_\lb$ via the map $m\mapsto (m'\mapsto [m,m'])$.
Since $M^{1,\lb}$ and $M^1_\lb$ are both $\cO$-modules of rank $1$, this map is therefore injective and induces the isomorphism 
$$C_0^\lb(M^1)=M^{1,\lb}/M^1_\lb \cong \cO/\psi(M^1_\lb)=\cO/([\delta_1,\delta_2]).$$
\end{proof}

\begin{cor} Assume that $T$ is Gorenstein. If $\lb^\vee:=\xi(\lb)\in T$ where $\xi$ is the isomorphism $\Hom_\cO(T,\cO){\cong}  T$, then we have
$$\eta_\lb=(\lb(\lb^\vee))$$
\end{cor}
\proof It is a direct consequence of the last assertion of the previous proposition using the perfect $\cO$-pairing $T\otimes T\rightarrow\cO$ given by the $T$-linear isomorphism
$T\cong \Hom_\cO(T,\cO)$.\qed

\subsection{Congruence modules for base change}\label{basechange}
We assume that $T$ and $T'$ are finite flat local $\cO$-algebras with $T_K$ and $T'_K$ semisimple as in section \ref{C0} and that we have an homomorphism  $\theta\colon T'\rightarrow  T$ 
such that $T'_K\rightarrow T_K$ is surjective.  We denote by $T^\sharp_K$ the semi-simple sub-algebra of $T'_K$ such that $T'_K\cong T_K\times T^\sharp_K$,
the first projection being given by $\theta$. We denote by $e_\theta$ the corresponding idempotent of $T^\prime_K$. Let $e^\sharp$ be the idempotent corresponding to the factor $T^\sharp_K$.
We fix $\lb\colon T\to \cO$
inducing $\lb'\colon T'\to \cO$ via $\theta$. Let $\c_\theta=\Ker(T^\prime\to T^\sharp)$.

For any $T'$-module $M$ flat over $\cO$, let $M^T=\e_\theta\cdot M$, $M_T=M\cap e_\theta M_K$, $M^{T^\sharp}=e^\sharp\cdot M$. We put
$$C_0^{\lb, \sharp}(M) :=  { M^{\lb'} /  (M_T)^\lb}$$
and we denote by $\eta^\sharp_\lb(M)$ its Fitting ideal.
We see that we have a canonical  surjective map:
$$ (M^{T^\sharp})\otimes_{\lb'}\cO\cong (M/M_T)\otimes_{\lb'}\cO\twoheadrightarrow { M^{\lb'} /  (M_T)^\lb}= C_0^{\lb, \sharp}(M)$$
which is an isomorphism if $M\otimes_{\lb'}\cO\cong M^{\lb'}$.  In particular if $M=T'$, we have 
 $$M_T= \frak c_\theta =\Ann_{T'}(\Ker\theta)$$ 
 and therefore 
$$C_0^{\lb, \sharp}(T')= T^\sharp\otimes_{T',\lb'}\cO=C_0^{\theta}(T')\otimes_{T',\lb'}\cO=\cO/\lb(\frak c_\theta)$$
and
$$\eta^\sharp_\lb:=\eta^\sharp_\lb(T')=\lb(\frak c_\theta).$$
We have the following lemma:
\begin{lem} \label{BC-congruenceN}
Let $M$ be a $T'$-module flat over $\cO$. Then
$$\eta_{\lb'}(M)=\eta_\lb(M_T).\eta^\sharp_\lb(M)$$

\end{lem}
\proof It follows from the definition that  $ (M_T)_{\lb}= M_{\lb'}$, we therefore have the following exact sequence:
\begin{equation}\label{excongid} 0 \rightarrow { (M_T)^\lb \over (M_T)_\lb} \rightarrow {  M^{\lb'} \over M_{\lb'} } \rightarrow { M^{\lb'}  \over  (M_T)^\lb} \rightarrow 0
\end{equation}
The claim then follows from the multiplicativity of Fitting ideals.\qed

\begin{cor} \label{Hida-factorization} We have the divisibility
$$\eta_{\lb'} \supset \eta_\lb .\eta_\lb^\sharp$$
If moreover $T$ and $T'$ are Gorenstein and if $\theta$ is surjective (Hida's formalism), then we have
$$\eta_{\lb'}=\eta_\lb .\eta^\sharp_\lb$$
\end{cor}
\begin{proof} 
We apply the previous lemma for $M=T'$ and notice that
$\eta_\lb(M_T)\supset \eta_\lb$ from \eqref{div-C0}.
Let's now prove the second assertion which was also proved by Hida in a slightly different way (\cite[Theorem 6.6]{H88} or \cite[Section 8.3]{HT17}).
We consider the following short exact sequence
$$0\rightarrow \Ker\theta\to T'\to T\to 0$$
By applying the functor $\Hom (\cdot,\cO)$, we get by using the Gorenstein-ness of $T'$ and $T$, the short exact sequence of $T'$-modules
\begin{equation}\label{ast}
0\rightarrow T\rightarrow T'\rightarrow \Hom_\cO(\Ker\theta,\cO)\to 0
\end{equation}
Observe that $M_T=T^\prime\cap T_K=\{x\in T^\prime; xT_K^\sharp=0\}$. Hence, if we let $\xi\colon T^\prime\cong\Hom_{\cO}(T^\prime,\cO)$,
we have
$$M_T=\{x\in T^\prime; \xi(x)\vert_{\Ker\,\theta}=0\}.$$
Therefore, the exact sequence \eqref{ast} provides the $T$-linear isomorphism $M_T\cong T$ and therefore $\eta_\lb(M_T)=\eta_\lb$ which implies the claim by applying the previous lemma.
\end{proof}
The following proposition will be used in 
Proposition \ref{Hida-linearform} and in 
Proposition \ref{propHidadual}  below.

\begin{pro}\label{linear-BC}
Let $M$ be a $T'$-module which is finite over $\cO$ and such that $M_K$ is of rank one over $T'_K$ and let $\Phi\in M^\ast=\Hom_\cO(M,\cO)$ such that
$\Ker (\Phi\otimes id_K)\supset M^\sharp_K$.  For all $\delta\in M_{\lb'}$, we have:
$$\Phi(\delta)\in\eta_\lb^\sharp(M^\ast)$$
\end{pro}

Note that that we don't assume that $M$ is torsion free over $\cO$ as this is automatic for its $\cO$-linear dual $N=M^\ast$.
\begin{proof}
To lighten the notations, let us put $N:=M^\ast$.  We consider the following commutative diagram
$$\begin{array}{ccccccccc}N&\to& N^T&\to &N^{\lb'}&{\cong}& \cO\\
\downarrow&&\downarrow&&\downarrow&&\downarrow\\
N^{T^\sharp}&\to& N^T/N_T&\to&N^{\lb'}/ (N_T)^\lb&\cong &\cO/\eta_{\lb}^\sharp(N)\end{array}$$
where the left bottom arrow comes from the isomorphism $N^T/N_T\cong N^{T^\sharp}/N_{T^\sharp}$.
Note that $N^{\lb'}\cong\cO$, hence $C_0^{\lb,\sharp}(N)=N^{\lb'}/N_T^{\lb}$ is $\cO$-cyclic. The composition 
$$ N^{\T^\sharp}\to N^T/N_T\to C_0^{\lb,\sharp}(N)=\cO/\eta_\lb^\sharp(N)$$
yields an $\cO$-linear homomorphism $ N^{T^\sharp}\to \cO/\eta_{\lb}^\sharp(N)$ 
which lifts to an $\cO$-linear form $ N^{T^\sharp}\to \cO$. 
Note that $N_K^T=(M_K^\sharp)^\perp$, hence by assumption, $\Phi\in N\cap N_K^T=N_T$. In other words, the image of $\Phi$ in $N^T/N_T$
is zero. It is therefore {\it a fortiori} zero in $N^{\lb'}/N_T^{\lb}\cong \cO/\eta_{\lb}^\sharp(N)$.

Now, since  $\delta\in M_{\lb'}$, it defines a surjective $\cO$-linear form  $N\to\cO$ by
$\delta^\vee\colon\phi\mapsto \phi(\delta)$ factorizing through $N\twoheadrightarrow N^{\lb'}$. 
We conclude that $\Phi(\delta)=\delta^\vee(\Phi)\in \eta_{\lb}^\sharp(N)$.
\end{proof}

\subsection{The congruence module $C_1$}
We keep the hypothesis and notations as before except that we only require the local algebra $T$ to be finite (not necessary flat) over $\cO$ 
and that $T_K$ is semi-simple. Let $\tilde T$ be the image of $T$ into $T_K$. We have $\tilde T=T/(\cO-tors)$.
Let $\lb\colon T\to\cO$ a surjective $\cO$-algebra homomorphism. 
We define $C_0^\lb(T)$ as $C_0^\lb({\tilde T})$. We write
$\wp_\lb:=\Ker(\lb)$. Then $\eta_\lb=\lb(\Ann_T(\wp_\lb))$. Recall that one defines
$$C_1^\lb(T):=\wp_\lb/\wp_\lb^2\cong  \Omega_{T/\cO}\otimes_{T,\lb}\cO$$

\begin{pro}(Wiles, Lenstra) \label{LCI}
Let $T$ and $\lb$ as above, then
$$\eta_\lb\supset \Fitt_\cO(C_1^\lb(T)).$$
Moreover if $T$ is a local $\cO$-algebra, then it is local complete intersection over $\cO$ if and only if the above inclusion is an equality.
\end{pro}
\begin{proof} We explain first the inclusion which is easy. Notice that $\tilde \wp_\lb=\Ker(\tilde T\stackrel{\lb}{\rightarrow}\cO)\subset \tilde T^c$ 
and that it is a faithful $\tilde T^c$ -module. We therefore have
$$\Fitt_\cO (\tilde T^c/\tilde\wp_\lb)\supset \Fitt_\cO(\tilde\wp_\lb\otimes (\tilde T^c/\tilde\wp_\lb))=\Fitt_\cO(C_1^\lb(\tilde T))\supset \Fitt_\cO(C_1^\lb( T))$$
by \cite[Appendix]{MW84}. We observe that the left hand side is equal to $\eta_\lambda$ because of the isomorphism $\tilde T^c/\wp_\lb\cong\cO/\eta_\lb$.
The second assertion is more difficult. We have a surjection $C_1^\lb( T)\to C_1^\lb(\tilde T)$, hence our assumption implies that 
$\Fitt_\cO(C_1^\lb(\tilde T))\supset\eta_\lb$. If $\tilde T$ was Gorenstein, \cite[Proposition 2, Appendix]{Wi95},
would imply that $\tilde T$ is local complete intersection over $\cO$ and that equality $\Fitt_\cO(C_1^\lb(\tilde T))=\eta_\lb$ holds. 
But the assumption of Gorensteiness for Proposition 2 of Appendix of \cite{Wi95} has been removed by Lenstra: the same result holds for $\tilde T$
without this assumption. 
Now, our assumption implies the equality  $\Fitt_\cO(C_1^\lb( T))=\Fitt_\cO(C_1^\lb(\tilde T))=\eta_\lb$; so by \cite[Proposition 1]{Wi95},
we have $\tilde T=T$ and this algebra is complete intersection over $\cO$.
\end{proof}

We put $\delta_\lb:=\Fitt_\cO(C_1^\lb(T))\eta_\lb^{-1}$ and call this ideal the Wiles defect of $T$ with respect to $\lb$. From the proposition above, it measures
how far $T$ is from being complete intersection over $\cO$. We have the following proposition which will be used in the proof of Theorem \ref{ThOPT}.

\begin{pro}\label{Wiles-defect}
We keep the same hypothesis as above and assume furthermore that $T$ has a presentation of the form
$$T\cong \cO[[x_1,\dots,x_g]]/(f_0,f_1,\dots,f_g)$$
with $f_0,f_1,\dots,f_g\in \cO[[x_1,\dots,x_g]]$. Then, 
$$\delta_\lb=\Fitt_\cO(\H_1(L_{T/\cO}\otimes_{T,\lb}\cO))$$
where $L_{T/\cO}$ stands for the cotangent complex of $T$ over $\cO$.
\end{pro}
Note that no assumption is made about the sequence $(f_0,f_1,\ldots,f_g)$ except that it has $g+1$ terms.
We start the proof by two Lemmas.

For any ideal $J$ of a noetherian ring $S$ and any finitely generated module $S$-module $M$, we denote by $\depth_J(M)$
the $J$-depth of $M$ (Recall that it is the maximal length of $M$-regular sequences in $J$) and by $\pd_S(M)$ the projective dimension
of $M$ over $S$, that is, the minimal length of a finite free resolution of $M$ over $S$.

\begin{lem} \label{regsequ} Let $(S,\m_S)$ be a local noetherian  regular ring which is a flat $\cO$-algebra. 
Let $J\subset S$ be an ideal such that $S/J$ is $\cO$-flat and regular, and let $M$ be an $S$-module such that $\depth_J(M)=g$ and $\pd_S(M)=1$. 
Then there exists an $M$-regular sequence $(y^\prime_1,\ldots,y^\prime_g)$ in $J$ such that 
$(\overline{y}_1^\prime,\ldots,\overline{y}_g^\prime)$ is free in the $g+1$-dimensional vector space
$J/\m_SJ$.
\end{lem}
\begin{proof} The proof is by induction on $g$. For $g=0$, there is nothing to prove.
Assume therefore that $g\geq 1$, in particular, $J\neq 0$ since  $\depth_J(M)=g\geq 1$.
Since $S/J$ is $\cO$-flat and regular,  $J$ is a prime ideal generated by a subset not containing $\varpi$ of a regular set of parameters for the regular local ring $S$ (see for example 
\cite[Theorem 36]{Mat80}).
In particular, $J$ is not contained in $\m_S^2$ and for any $x\in J\backslash J\cap\m_S^2$, $S/xS$ is regular and $\cO$-flat.
In fact, we have $J\cap\m_S^2=\m_SJ$. Indeed, consider the exact sequence
$$0\to J\to\m_S\to\m_S/J\to 0$$
 After tensoring it over $S$ by the residue field $\kappa$ of $S$, we get
\begin{equation}\label{exact2}
J/\m_SJ\stackrel{(*)}{\longrightarrow}\m_S/\m_S^2\to\m_S/J+\m_S^2\to 0
\end{equation}
Let $x_1,\dots,x_d$ a regular  system of parameters for $S$ such that $J=(x_1,\dots, x_r)$. Then $\m_S/J$ is generated by $x_{r+1},\dots, x_d$. Therefore 
the dimensions over $\kappa$ of the left and right terms of the exact sequence \eqref{exact2} are at most respectively $r$ and $d-r$. Since $S$ is regular, the dimension of the middle term is $d$, therefore this forces $(*)$ to be injective and thus $J\cap\m_S^2=\m_SJ$.

\medskip
Let us now return to the proof of our result by induction and assume  that our result holds true for $g-1$. Let
$$0\to S^n\to S^m\to M\to 0$$
be a length one free resolution of $M$.
It is sufficient to show that we can find an $M$-regular element in $x\in J\backslash \m_S J$. 
Indeed if $x\in J\backslash  \m_S J$ is $M$-regular, $\Tor_1^S(M,S/xS)=0$ and therefore 
$0\rightarrow (S/xS)^n\rightarrow (S/xS)^m\rightarrow M/xM\rightarrow 0$
is exact and therefore $\pd_{S/xS}(M/xM)=1$. Now as we have observed before, $S/xS$ is regular of dimension $g+1$ and is $\cO$-flat by our choice of $x$.
Moreover $\depth_{J/xS}(M/xM)=g-1$ since $x$ will be $M$-regular.
Thus we will be able to conclude by the induction hypothesis applied to  
$S/xS $ and $M/xM$.

The $M$-regular elements 
are those 
which are not contained in any prime ideals of $\Ass_S(M)=\{P_1,\dots,P_q\}$ which is the finite set of associate ideals 
of the finite $S$-module $M$. 
Since $\depth_J(M)=g\geq 1$, we know that $J$ is not contained in $\cup_{j=1}^qP_j$, and we just need to show that $J\backslash (J^2+\varpi J)$ 
is not contained in $\cup_{j=1}^qP_j$. 

 Let us prove our claim by contradiction by assuming that  $J\backslash   \m_S J\subset \cup_{j=1}^qP_j$ and choose $x\in  J$ which is not in $  \m_S J$.
Now let us take  any $z\in  \m_S J$. For any positive integer $n$, we have $x+z^n\in J\backslash   \m_S J\subset   \cup_{j=1}^qP_j$. 
Therefore there exists $j_0$ such that $x+z^n\in P_{j_0}$ for infinitely many $n>0$. This implies that 
$z^n(1-z^m)\in P_{j_0}$ for two positive integers $n$ and $m$. But $1-z^m$ is invertible because $z$ 
belongs to the maximal ideal of $S$.
Since $P_{j_0}$ is a prime ideal, this forces $z$ to be in $P_{j_0}\subset   \cup_{j=1}^qP_j $. 
We have therefore proved that $\m_S J\subset \cup_{j=1}^qP_j$ and thus $J\subset \cup_{j=1}^qP_j$ 
which is a contradiction since $\depth_J(M)=g\geq 1$.
\end{proof}

\begin{lem} We keep the hypotheses of the proposition. Then we have an isomorphism
$$T\cong T_0/(f)$$
where $T_0$ is a $\cO$-finite flat local complete intersection and $f\in T_0$.
\end{lem}
\begin{proof} 
Let us consider $R:=\cO[[x_1,\dots,x_g]]$ as an $S:=\cO[[y_0,\dots,y_g]]$-algebra via the morphism $\varphi$ sending 
$y_i$ to $f_i$ for $i\in\{0,\dots,g\}$.
Let $I_S=(y_0,\dots,y_g)$ and $I_R=(f_0,\dots,f_g)=\varphi(I_S)$. Since $R$ is regular, we have
$$\depth_{I_S}(R)=\depth_{I_R}(R)=\dim R-\dim R/I_R=1+g-1=g$$
(recall $R/I_R=T$ is $1$-dimensional).
Therefore we deduce that there exist $y'_1,\dots, y'_g\in I_S$ such that the sequence $(y'_1,\dots,y'_g)$ is $R$-regular.
Moreover, by Lemma \ref{regsequ} with $J=I_S$ and $M=R$, we may assume that these elements are not divisible by $\varpi$ and that their image in $I_S/\m_S I_S$ 
is a part of an $\cO/\varpi\cO$-basis
of this quotient. Therefore the $y'_1,\dots,y'_g$ can be completed into a system of generators $y'_0,\dots,y'_g$ for $I_S$. 
Let $f'_i$ be the image of $y'_i$ in $R$. Then we see that $T_0:=R/(f'_1,\dots,f'_g)$ is complete intersection over $\cO$
since the $y'_1,\dots,y'_g$ form an $R$-regular sequence, and $T=T_0/(f)$ with $f$ the image of $y'_0$ in $T_0$.

\end{proof}

We now return to the proof of Proposition \ref{Wiles-defect}. 

\begin{proof}

Since $T_0$ is local complete intersection it is in particular Gorenstein. Using the invariant pairing on $T_0$, we see that
\begin{eqnarray}\label{duality}
\Hom_\cO(T,\cO)=\Hom_\cO(T_0/(f),\cO)\cong \Ker( T_0\stackrel{\times f}{\rightarrow }T_0)=\Ann_{T_0}(f) 
\end{eqnarray}
We consider the base change situation $\lb_0\colon T_0\rightarrow \tilde T \stackrel{\lb}{\rightarrow}\cO$ 
(we write $\tilde T= T/(\cO-tors)$ as at the beginning of this section). Similarly to the exact sequence \eqref{excongid} in the proof of
Lemma  \ref{BC-congruenceN}, we have the exact sequence
$$
0\rightarrow {(\Ann_{T_0}(f))^\lb \over (T_0)_\lb}\rightarrow {( T_0)^\lb\over (T_0)_\lb}\rightarrow {(T_0)^\lb \over \Ann_{T_0}(f))^\lb}\rightarrow 0
$$
and therefore
\begin{equation}\label{BCCN}
\eta_{\lb_0}(T_0)=\eta_\lb(\Ann_{T_0}(f))\cdot\lb(\Ann_{T_0}(f))= \eta_\lb \cdot\lb(\Ann_{T_0}(f))
\end{equation}
since $ \eta_\lb(\tilde T)= \eta_\lb(\Ann_{T_0}(f))$ from \eqref{duality} and Proposition \ref{Pontryaginprop}.

For any ring $A$ and $B$ an $A$-algebra, we denote by $L_{B/A}$ the cotangent complex of $B$ over $A$. The distinguished triangle associated to $\cO\rightarrow T_0\rightarrow T$,
 gives by tensoring with $\cO$ over $\lambda\colon T\to \cO$ (resp. $\lambda_0\colon T_0\to \cO$)
 another distinguished triangle
\begin{eqnarray*}
L_{T_0/\cO}\otimes_{\lambda_0,T_0}\cO\rightarrow L_{T/O}\otimes_{\lambda,T}\cO\rightarrow L_{T/T_0}\otimes_{\lambda,T}\cO\rightarrow L_{T_0/O}\otimes _{\lambda_0,T_0}\cO[1]
\end{eqnarray*}
providing the long exact sequence
\begin{eqnarray*}  \H_1(L_{T_0/\cO}\otimes_{\lb_0,T_0}\cO)\rightarrow \H_1( L_{T/\cO}\otimes_{\lb,T}\cO)\rightarrow
 \qquad \qquad\qquad \qquad \qquad  \\  \qquad \qquad\qquad\rightarrow  
(f)/(f^2)\otimes _{T_0,\lb_0}\cO
\rightarrow 
\Omega_{T_0/O}\otimes_{\lb_0,T_0}O\rightarrow \Omega_{T/\cO}\otimes_{\lambda,T}\cO\rightarrow 0
\end{eqnarray*}

Let us show $ \H_1(L_{T_0/\cO}\otimes_{\lb_0,T_0}\cO)=0$.
Since $T_0$ is local complete intersection $T_0\cong \cO[[x_1\dots,x_g]]/(f_1,\dots,f_g)$ 
where $(f_1,\dots, f_g)$ is a regular sequence generating an ideal $I$ and
$L_{T_0/\cO}$ is quasi-isomorphic to the perfect complex $[I/I^2\rightarrow \Omega_{ \cO[[x_1\dots,x_g]]/\cO}]$ 
with $I/I^2$ is placed in degree $-1$. 
In particular, one sees that
$ \H_1(L_{T_0/\cO}\otimes_{\lb_0,T_0}\cO)$ is the kernel of the canonical map
$$I/I^2\otimes_{\lb_0,T_0}\cO\rightarrow \Omega_{ \cO[[x_1\dots,x_g]]/\cO}\otimes\cO$$
which is injective since we know that the source and target are free $\cO$-module of rank $g$ and that its cokernel 
is $C_1^{\lb_0}(T_0)$ which is $\cO$-torsion. 
We deduce that
we have the exact sequence:
\begin{eqnarray*}  0\rightarrow \H_1( L_{T/\cO}\otimes_{\lb,T}\cO)\rightarrow (f)/(f^2)\otimes _{T_0,\lb_0}\cO
\rightarrow 
\Omega_{T_0/O}\otimes_{\lb_0,T_0}O\rightarrow \Omega_{T/\cO}\otimes_{\lambda,T}\cO\rightarrow 0
\end{eqnarray*}
Since $(f)/(f^2)\cong T_0/\Ann_{T_0}(f)$, we deduce that $\Fitt_\cO((f)/(f^2)\otimes _{T_0,\lb_0}\cO)=\lb(\Ann_{T_0}(f))$.
and (since all the modules in the exact sequence are finite) we deduce
\begin{eqnarray} \label{CCT}
\Fitt_\cO C_1^\lb(T)\cdot \lb(\Ann_{T_0}(f)) = \Fitt_\cO C_1^{\lb_0}(T_0)\cdot \Fitt_\cO \H_1(L_{T/\cO}\otimes_{\lb,T}\cO)
\end{eqnarray}
Since $T_0$ is local complete intersection, by Proposition \ref{LCI} we know that $\eta_{\lb_0}(T_0)=\Fitt_\cO C_1^{\lb_0}(T_0)$. 
By combining  \eqref{CCT} and \eqref{BCCN}, our claim follows.
\end{proof}

\medskip
We end this section by a definition and a lemma which shows its importance.
As in the beginning of Section \ref{basechange}, we give ourselves a "base change datum"
$T'\stackrel{\theta}{\rightarrow} T\stackrel{\lambda}{\rightarrow} \cO$ where $\theta\colon T^\prime\to T$ is a homomorphism  between two finite flat local $\cO$-algebras 
with $T_K$ and $T'_K$ semisimple 
such that $T'_K\rightarrow T_K$ is surjective and where  $\lb\colon T\to \cO$ is an homomorphism of $\cO$-algebras. We denote again $\lb^\prime:=\lb\circ\theta$.

\begin{de}\label{C1BCde}  The base change higher congruence module $C_1^{\lambda,\sharp}$ is defined as
$$C_1^{\lambda,\sharp}=\ker\theta \otimes_{\lb'}\cO.$$
\end{de}
With this definition, we have

\begin{lem}\label{C1BCle} Assume that we have a base change datum $T'\stackrel{\theta}{\rightarrow} T\stackrel{\lambda}{\rightarrow} \cO$ 
such that $T$ is local complete intersection.
We have the following short exact sequence
$$0\rightarrow C_1^{\lambda,\sharp} \rightarrow C_1^{\lb'}(T')\rightarrow C_1^{\lb}(T)\rightarrow 0.$$
in particular, we have
$$\Fitt_\cO(C_1^{\lb'}(T'))=\Fitt_\cO (C_1^{\lb}(T))\cdot \Fitt_\cO (C_1^{\lambda,\sharp})$$
\end{lem}
\begin{proof}
We have the distinguished triangle 
\begin{eqnarray}\label{triangle}
L_{T'/\cO}\otimes_{\lambda',T'}\cO\rightarrow L_{T/O}\otimes_{\lambda,T}\cO\rightarrow L_{T/T'}\otimes_{\lambda,T}\cO\rightarrow L_{T'/O}\otimes _{\lambda',T'}\cO[1]
\end{eqnarray}
Since $T$ is complete intersection, we have as in the previous proof $ \H_1(L_{T/\cO}\otimes_{\lb,T}\cO)=0$, therefore taking the cohomology of the triangle \eqref{triangle}, we have the short exact sequence
$$0\rightarrow {\Ker\theta \over (\Ker\theta)^2}\otimes_{\lb',T'}\cO  \rightarrow C_1^{\lb'}(T')\rightarrow C_1^{\lb}(T)\rightarrow 0$$
Since ${\Ker\theta \over (\Ker\theta)^2}\otimes_{\lb',T'}\cO =\Ker\theta \otimes_{\lb',T'}\cO $, our claim follows.
\end{proof}

\section{Congruence ideals associated to classical cuspforms} \label{sectManinCongruence}

The purpose of this section is to recall some well-known facts about the congruence ideals associated to cuspidal eigenforms as a pretext to fix some notations.
Let us fix an odd prime $p$ and an isomorphism $\iota_p\colon \C\cong\overline{\Q}_p$ so
 that every complex number can be considered as an element of $\overline\Q_p$ via this isomorphism.
 If $A$ is an ideal of a finite extension of $\Z_p$ and $B$ is an element of $\overline\bQ_p$, we will write $A\sim B$ if a 
 generator of $A$ has same $p$--adic valuation as $B$. If $A$ and $B$ are two complex numbers
 then $A\sim B$ will mean that either both are zero or that their quotient is a $p$-adic unit.

\medskip
Let $f$ be a primitive cuspidal eigenform of weight $k\geq 2$ and level $N$ with trivial Nebentypus.
This hypothesis on the Nebentypus
could be relaxed if a slight generalization of Cornut-Vatsal results was proven. 
$$f(z)=\sum_{n=1}^\infty a_n q^n\hbox{ with }q=e^{2i\pi z}$$
for all $z$ in the Poincar\'e upper half-plane $\h$.
Throughout this work, we assume that $p>k-2$ and that $N$ is prime to $p$. %
Let $K\subset \overline{\Q}_p$ be a $p$-adic field
containing the
Hecke eigenvalues of $f$. We denote by  $\cO$ its valuation ring with uniformizer parameter $\varpi$ and residue field $\F=\cO/\varpi\cO$.
Let $h_k(N;\cO)$ be the $\cO$-algebra generated by the Hecke operators outside $N$ acting on the space of cusp forms of weight $k$ for $\Gamma_0(N)$.
 Let $\m$ be the maximal ideal of $h_k(N;\cO)$ 
associated to the eigensystem $\lb_f\colon h_k(N;\cO)\rightarrow \cO$ of $f$ modulo $\varpi$ and denote by $\T$ the localization of $h_k(N;\cO)$ at $\m$. 
We have therefore a homomorphism $\lambda_f\colon\T\to\cO$ sending Hecke operators to their eigenvalues on the eigenvector $f$.

\medskip
Manin has defined periods $\Omega_{f}^{\pm}\in \C^\times/\cO^\times$ associated to $f$. We now recall their definitions.
Let $\cV_k$ be the local system over the modular curve $Y_0(N)(\C)$ associated to the representation of $\Gamma_0(N)$ 
on the $\cO$-module of homogeneous polynomials of degree $k-2$ in two variables $X,Y$, 
for the action $(\gamma\cdot P)(X,Y)=P((X,Y)\cdot{}^t\gamma^{-1})$.
Let $\H$ be the torsion-free quotient of the interior Betti cohomology $\H^1_{!,B}(Y_0(N),\cV_k(\cO))$ 
(as usual, interior cohomology refers to the image of $\H^1_c\to\H^1$). 
By the Eichler-Shimura isomorphism, the $\lambda_f$-isotypic submodule $\H[\lambda_f]$ of $\H$
is free of rank $2$ over $\cO$.
Since $p$ is odd, the complex conjugation acting on the pair $(Y_0(N)(\C),\cV_k)$ provides a decomposition $\H=\H^+\oplus \H^-$;
the intersection $\H[\lambda_f]^{\pm}$ of $\H^{\pm}$ with $\H[\lambda_f]$ 
is a free $\cO$-module of rank one. 
Let $\delta_f^{\pm}$ be an $\cO$-basis of this module. Let $f^c=\overline{f(-\overline{z})}$
The cusp forms $f$ and $f^c$ provide two de Rham cohomology classes $\omega_f=f(z)(X-zY)^{k-2}dz$ and $\overline{\omega}_{f^c}=\overline{f^c(z)}(X-\overline{z}Y)^{k-2}d\overline{z}$, 
both in $\H[\lambda_f]\otimes_{\cO} \C$.
Let
$$\omega_f^{\pm}={1\over 2}\cdot({\omega}_{f}\pm \overline{\omega}_{f^c}).$$
Since $\omega_f^{\pm}\in \H[\lambda_f]^{\pm}\otimes_{\cO}\C$, there exists a unique $\Omega_f^{\pm}\in\C^\times/\cO^\times$ such that $\omega_f^{\pm}=\Omega_f^{\pm}\cdot\delta_f^{\pm}$.
 These complex number are called the Manin periods of $f$. They correspond to the Deligne periods of the dual of the motive $M_f$ associated to $f$.
  
 \medskip 
We now recall the definition of the Hecke congruence ideal $\eta_{f}$.
We have a splitting over $K$:
$\T\otimes_\cO K=K\times \T^c_K$ where the first projection is given by $\lambda_f$
and $\T^c_K$ is a finite $K$-algebra. 
Let $\T^c=\im(\T\to \T^c_K)$ be the image of $\T$ by the second projection.
We have an injective $\cO$-algebra homomorphism $\T\hookrightarrow \cO\times \T^\c$.
We denote by $e_\cO=(1,0)$ and $e_{\T^c}=(0,1)$ the idempotents associated to the cartesian product on the right-hand side.
Let 
$$C_0(f):=C_0^{\lb_f}(\T) =  (\cO\times \T^c)/\T=\cO\otimes_{\T} \T^c$$
be the Hecke congruence module. By its second description, it is an $\cO$-algebra which is a quotient of $\cO$ 
(and of $\T^c$).
We define the "Hecke" congruence ideal $\eta_{f}$ (defined up to multiplication by an element of $\cO^\times$) 
as a generator of the ideal of $\cO$ 
given by $\T\cap(\cO\times\{0\})$. It is also the denominator of $e_\cO$ in $\T$. We have
$$C_0(f)=\cO/\eta_{f}.$$
We can also define  the cohomological congruence module associated to $M=\H^{\pm}_\m$ (or, equivalently, to $\H^{\pm}$).
We put
$$M^{\cO}=e_\cO\cdot \H_{\m}^\pm \quad\mbox{\rm and}\, M_{\cO}=\H_{\m}^\pm\cap e_\cO\cdot \H_{\m}^\pm[{1\over p}]=\H^\pm_{\m}[\lambda_f].$$
These $\cO$-modules are free of rank $1$ and $\delta^\pm_f$ is a basis of $M_\cO$. From the Hecke-equivariant perfect duality 
(see for instance \cite[Propositions 2.1 and 3.3]{H81}) between $\H^+_\m$ and $\H^-_{\m}$,
we know from Proposition \ref{Pontryaginprop} that  the  resulting  congruence module is independant of the sign $\pm$  and we put
$$C_0^{coh}(f):=C_0^{\lb_f}(M)=M^{\cO}/ M_{\cO}\cong \cO/\eta_f^{coh}$$
with  $\eta_f^{coh}=[\delta_f^+,\delta_f^-]$ where $[\bullet,\bullet]$ stands for the Poincar\'e duality pairing 
twisted by the Atkin-Lehner involution
(see for instance \cite[Proposition 2.1 and 3.3]{H81}.
It is called the "cohomological" congruence ideal associated to $f$
(well defined up to an element of $\cO^\times$).
If $\H^{\pm}_{\m}$ is free over $\T$, then we have  
$$\eta_{f}^{coh}=\eta_{f}.$$
This congruence ideal can be computed in terms of the Adjoint $L$-function associated to $f$. 
Before introducing it and stating the main result due to Hida, we recall
some useful conditions on the $p$-adic Galois representations associated with $f$.

\medskip
Let $\rho_f\colon\Gamma_\Q\to\GL_2(\cO)$ be 
 the $p$-adic Galois representation associated to $f$ and let $\overline{\rho}_f\colon\Gamma_\Q\to\GL_2(\F)$
its reduction modulo $\varpi$. Let $E$ be a number field. We will consider the following hypotheses:
\medskip
\begin{itemize}
\item[(Irr${}_E$)] The restriction of  $\overline\rho_f $ to $\Gamma_E$ is absolutely irreducible
\end{itemize}
\medskip
For $E=\Q$, we will just write $(\Irr)$.
Recall \cite[Theorem 2.1]{Wi95}  that under $(\Irr)$ and $p>k$, the module $\H_\m$ is free 
over the localized Hecke algebra $\T$ and therefore $\T$ is Gorenstein.

\medskip

For any place $v$ of $E$ dividing $N$, let $I_v\subset\Gamma_E$ be an inertia subgroup. 
The following condition will be also considered:

\begin{itemize}
\item[(Min${}_E)$] $\rho_f|_{\Gamma_E}$ is a minimal deformation lift of $\overline \rho_f|_{\Gamma_E}$: for any place $v$ of $E$
dividing $N$, if $\ord_\ell(N)=1$, we assume $\overline{\rho}_f\vert_{I_v}\neq 1$ and 
If $\ord_v(N)>1$, we assume that the reduction induces an isomorphism ${\rho}_f(I_v)\cong\overline{\rho}_f(I_v)$ and 
the latter group acts irreducibly on $k^2$. 
\end{itemize}

We simply write $(\Min)$ if $E=\Q$. 
%
%
The following assumptions will be in force throughout the paper as it will be used to apply a result of Cornut-Vatsal and its 
generalization by Chida-Hsieh.
\begin{itemize}
\item[(CV)] There exists a prime $\ell$ such that the restriction of $\overline \rho_f|_{\Gamma_{\Q_\ell}}$ 
to a decomposition subgroup at $\ell$ is indecomposable. If there is a prime $q|N$ such that $q^2|N$ and $q\equiv -1\pmod p$, then we assume that $\overline{\rho}_f$ restricted to the Inertia subgroup at $q$ is irreducible.

\end{itemize}
For example, the condition (CV) holds if there is a prime $\ell$ such that $\ord_\ell(N)=1$ and $(\Min)$ holds at $\ell$. 

\medskip
Let $\Ad(\rho_f)\subset\rho_f^\vee\otimes\rho_f$ be the adjoint representation of $\rho_f$ acting on the space of trace zero endomorphisms. 
For any Dirichlet 
character $\alpha$, we consider the imprimitive adjoint $L$-function defined as follows.
\begin{eqnarray*}
L(\Ad f\otimes\alpha,s):=\prod_{\ell\not{|}N} \det(Id-\Ad(\rho_f)(\Frob_\ell)\alpha(\ell))\ell^{-s})^{-1}  \times\qquad\\
\qquad \times \prod_{\ell|N}\left\{ \begin{array}{l@{\quad}l}
 (1-\alpha(\ell)\ell^{-s}) ^{-1} &  \hbox{if $f$ is principal series at $\ell$}\\
(1-\alpha(\ell)\ell^{-s-1}) ^{-1} &  \hbox{if $f$ is special at $\ell$}
\end{array}\right.
\end{eqnarray*}
where $\Frob_\ell$ stands for a geometric Frobenius at $\ell$. When $\alpha$ is trivial, we just write $L(\Ad(f),s)$. Let $\varphi(N):=\# (\Z/N\Z)^\times$, 
the following theorem links the congruence ideals we have defined above to the special value of $L(\Ad f,s)$ at $s=1$.

\begin{thm} \label{classical}a) For $p$ a prime not dividing $6N\varphi(N)$ such that $p>k-2$, we have 
$${L(\Ad\,f,1)\over\pi^{k+1}\Omega_f^+\Omega_f^-}\sim \eta_f^{coh}$$
b)If moreover $(\Irr)$ and $p>k$ hold, then we have
$${L(\Ad\,f,1)\over\pi^{k+1}\Omega_f^+\Omega_f^-}\sim \eta_f.$$
\end{thm}
\begin{proof}
a) is due to Hida  (see \cite[Theorem A]{H81} and \cite{H88}) and follows from the computation of 
$[\delta_f^+,\delta_f^-]={[\omega_f^+,\omega_f^-]\over \Omega^+_f\Omega^-_f}$ in terms of the Petersson norm of $f$ which
is closely related to the adjoint L-value at $1$ by a result due to Shimura. We refer to Hida's papers for the precise computations. 
For b), under these hypothesis, it is known from \cite{Wi95} using an argument essentially due to Mazur
that $\H_{\m}$ is free over $\T$ and therefore
$\eta_f=\eta_f^{coh}$, so the result follows from a).
\end{proof}

\section{The case of totally real fields}
\subsection{Definition of periods and the integral period relation conjecture}\label{periodsF}

Let $F$ be a totally real field of degree $d$, discriminant $D$, $\cO_F$ its ring of integers and $I_F=\Hom_{\alg}(F,\overline{\Q})$ 
be the set of embeddings of $F$. 
Let $t=\sum_\sigma\sigma\in\Z[I_F]$. 
Let $\widehat{\cO}_F=\cO_F\otimes\widehat{\Z}$, $F_f=\widehat{\cO}_F\otimes\Q$,
$F_\infty=F\otimes\R$, and $\bA_F=F_f\times F_\infty$ be the ring of ad\`eles of $F$.
Let $\h$ be the upper-half plane. The group $\GL^+_2(F_\infty)$ of matrices with totally positive determinant acts transitively on $\h^{I_F}$ and
we have $\GL^+_2(F_\infty)/F_\infty^\times\SO_2(F_\infty)\cong \h^{I_F}$ via $g\mapsto g(\underline{i})$ where 
$\underline{i}=(\sqrt{-1},\ldots,\sqrt{-1})$. 
Let $k\geq 2$ and $\n$ an ideal of $\cO_F$ such that $(N)=\Z\cap\n$. 
Let $$U_0(\n)=\{\left(\begin{array}{cc} a&b\\c&d\end{array}\right)\in \GL_2(\widehat{\cO}_F); c\equiv 0\pmod{\n}\}$$
For any subset $I\subset I_F$, let
$S_k^F(I,\n)$ be the space of adelic Hilbert cuspforms of weight $k$ and level $U_0(\n)$ 
which are holomorphic with respect to the variables $z_\tau\in \h$, $\tau\in I$ and antiholomorphic 
with respect to the variables $z_\tau$ for $\tau\notin I$. We simply write $S_k^F(\n)$ for $S_k^F(I_F,\n)$.
Let ${\ff}\in S_k^F(\n)$ be a normalized cuspidal newform. 
Let $K$ be a finite extension of $\Q_p$ containing the Galois closure of $F$ and the Hecke eigenvalues of $\ff$.
Let $\cO$ be its valuation ring. 


We assume $\n\cap\Z=N\Z$ with $N>3$. Then the congruence subgroup
 $$U(\n)=\{\left(\begin{array}{cc} a&b\\c&d\end{array}\right); c\equiv 0\pmod{\n}, d\equiv 1\pmod{\n}\}$$
is neat and normal in $U_1(\n)$. Let $\varphi_F(\n)=[U_0(\n):U_1(\n)]=\Card (\cO_F/\n)^\times=N(\n)\prod_{\lambda\vert N}(1-N(\lambda)^{-1})$.
We assume that $p$ is prime to $N(\n)\varphi_F(\n)$, thus it is also prime to $[U_0(\n):U(\n)]$.
Let
$$X_F=\GL_2(F)\backslash (\GL_2(F_f)/U_0(\n)\times \h^{I_F})$$
be the Hilbert modular orbifold of level 
$U_0(\n)$ and
$$Y_F=\GL_2(F)\backslash (\GL_2(F_f)/U(\n)\times \h^{I_F})$$
be the Hilbert modular variety of level 
$U(\n)$.
Let $\cV^F_k$ the locally constant sheaf in $\cO$-modules on $Y_F$ associated to the module 
$\bigotimes_{\tau\in I_F}\cO[X_\tau,Y_\tau]_{k-2}$ where $\cO[X_\tau,Y_\tau]_{k-2}$ denotes the 
$\GL_2(\cO_F)$-module of homogeneous polynomials of degree $k-2$ in two variables. 
We can put 
$$\H^d_?(X_F,\cV^F_k)=\H^0(U_0(\n)/U(\n),\H^d_?(Y_F,\cV^F_k(\cO)))$$
for $?=\emptyset,c,!$ (with the usual notation for the interior cohomology $\H^ \bullet_!=\im(\H^\bullet_c\to\H^\bullet)$).
Let us write for short $\H^\bullet(F)=\H^\bullet(X_F,\cV_k(\cO))$,  $\H_c^\bullet(F)=\H_c^\bullet(X_F,\cV^F_k(\cO))$ and 
$\H_!^\bullet (F)=\im (\H_c^\bullet(X_F,\cV^F_k))\to\H^\bullet(X_F,\cV^F_k)))$.

\medskip
When $\n=N. O_F$ and $F$ is Galois over $\Q$, the Galois group $Gal(F/\Q)$  acts canonically on the 
moduli problem defining the Hilbert-Blumenthal varieties of level $U_0(\n)$ and $U(\n)$ so that the induced action 
on $X_F=\GL_2(F)\backslash \GL_2(\F_f)/U_0(N.O_F)\times \h^{I_F}$ is the usual Galois action on the finite adeles and 
permutes the different copies of $\h$ via its canonical action on $I_F$. 
Moreover, the induced action of $Gal(F/\Q)$ on the \' etale cohomology with coefficient in $\cV^F_k(\cO)$ commutes
 with the Galois action of $G_F$. However this action, does not commute with the Hecke action.

\medskip
Recall that we have the Hecke-equivariant Harder-Matsushima-Shimura isomorphism \cite[Section 3.6]{Har87}:
$$\bigoplus_{I\subset I_F}S_k^F(I,\n)\cong \H_!^\bullet (F)\otimes_{\cO}\C$$
given for each subset $I\subset I_F$ by 
$${\ff_I\mapsto \omega(\ff_I):= 
 \ff_I(z) (X_I-z_I Y_I)^{k-2} \cdot
(X_{\overline{I}}-\overline{z}_{\overline{I}}Y_{\overline{I}})^{k-2} dz_I\wedge d\overline{z}_{\overline{I}}}$$
where $(X_I-z_IY_I)^{k-2}=\prod_{\tau\in I}(X_\tau-z_\tau Y_\tau)^{k-2}$ and
$(X_{\overline{I}}-\overline{z}_{\overline{I}}Y_{\overline{I}})^{k-2}=
\prod_{\tau\notin I}(X_\tau-\overline{z}_\tau Y_\tau)^{k-2}$.

\medskip
Let $(w_\tau)_{\tau\in I_F}$ be the canonical basis of the Weyl group $W_F=\{\pm 1\}^{I_F}$: $w_\tau=(w_{\tau,\sigma})_{\sigma\in I_F}$
with
$w_{\tau,\sigma}=1$ unless $\sigma=\tau$ and $w_{\tau,\sigma}=-1$. For $I\subset I_F$, we write $w_I=\prod_{\tau\in I}w_\tau$.
The group $W_F$ acts on $\h^{I_F}$ as follows: 
$w_I\cdot (z_\tau)=(z^\prime_\tau)$ where $z^\prime_\tau=-\overline{z}_\tau$ 
for $\tau\in I$, resp.
$z^\prime_\tau=z_\tau$ for $\tau\notin I$. This action extends to an action on $(X_F,\cV^F_k(\cO))$ 
given on a section $s$ of
$\cV_k(\cO)$ by $w\bullet s= w\cdot w^\ast s$. 

 It will be useful to notice also that
when $F$ is Galois over $\Q$, that the action of $Gal(F/\Q)$ and $W_F$ do not commute but are related by the relation:
\begin{equation}
\sigma\circ w_\tau\circ\sigma^{-1}=w_{\tau\circ\sigma}\qquad \forall \sigma\in Gal(F/\Q),\; \forall\tau\in I_F
\end{equation}

\medskip
For $\epsilon\in\widehat{W_F}$, let $\H_!^\epsilon (F)$ be the $\epsilon$-eigenspace of $\H_!^d (F)$.
We have
$$\H_!^d (F)=\bigoplus_{\epsilon\in \widehat{W_F}}\H_!^\epsilon (F)$$
Let $h_k^F=h^F_k(U_0(\n),\cO)$ be the Hecke algebra acting on $H^d_!(F)$. 
By the Harder-Matsushima-Shimura isomorphism, 
we have a homomorphism $\lambda_{\ff}\colon h_k^F\to\cO$ giving the system of Hecke eigenvalues for our cuspidal newform  $\ff$.
We denote by  $\m_F=\Ker\overline\lb_\ff$ be the maximal ideal of $h_k^F$ where $\overline{\lambda}_{\ff}$ 
denotes the reduction of $\lb_\ff$ modulo the uniformizer of $\cO$.

We assume that $\m_F$ is non-Eisenstein, by which we mean that the degree $2$ residual Galois representation $\overline{\rho}_{\ff}$ 
associated to $\ff$ is absolutely irreducible.

Let $\T_F$ denote the $\m_F$-adic completion of $h_k^F$. 
Let  $\H_!^d (F)[\lambda_{\ff}]$, resp. $\H_!^\epsilon (F)[\lambda_{\ff}]$ be the $\lambda_{\ff}$-isotypic 
part of the $\cO$-module $\H_!^d (F)$ resp. $\H_!^\epsilon (F)$. 
From the description of the cuspidal cohomology, it follows that
$$\H_!^d (F)_{\m_F}\otimes \Q_p\cong( \T_F\otimes \Q_p)^{2^d}$$
In particular, up to torsion,
$\H_!^\epsilon (F)[\lambda_{\ff}]$ is of rank $2^d$ over $\cO$ 
and for each $\epsilon \in \widehat{W_F}$, $\H_!^\epsilon (F)[\lambda_{\ff}]$ is a rank one $\cO$-module (up to torsion). 
Let $\ff\in S_k^F(I_F,\n)$ and 
$\omega^\epsilon(\ff)=2^{-d}\cdot\sum_{w\in W_F}\overline{\epsilon}( w)\cdot w^\ast\omega(\ff)$.
We see that $\omega^\epsilon(\ff)$ belongs to the $\epsilon$ component of $\H^d_{dR}(X_F,\cV_k(\C))$. 
Hence by de Rham comparison theorem it can be viewed in $\H_!^\epsilon (F)\otimes_{\cO}\C$.
More precisely, 
For each $\epsilon \in \widehat{W_F}$, it follows from our assumptions that $\H_!^\epsilon (F)[\lambda_{\ff}]$ is a rank one $\cO$-module.
Let $\delta_{\ff}^\epsilon$ be an $\cO$-basis of this module up to torsion.
As $\omega^\epsilon(\ff)$ belongs to $\H_!^\epsilon (F)[\lambda_{\ff}]\otimes_{\cO}\C$, we can make the following definition.

\begin{de} For each $\epsilon\in\widehat{W_F}$, the $\epsilon$-period $\Omega_{\ff}^\epsilon\in\C^\times/\cO^\times$
of $\ff$ is defined as the unique element such that 
$$\omega^\epsilon(\ff)=\Omega_{\ff}^\epsilon\cdot \delta_{\ff}^\epsilon.
$$
\end{de}


Let $\ff=f_F$ be the normalized holomorphic base change from the normalized newform $f\in S_k(\Gamma_0(N))$
with the notations of the previous section. We also denote by $\varphi_F(N):=\# (O_F/N O_F)^\times$.
For $\epsilon\in\widehat{W}_F$, let $D_\epsilon^{\pm}=\{\tau\in I_F; \epsilon(w_\tau)=\pm 1\}$ and $d^\pm=d_\epsilon^{\pm}=\# D_\epsilon^{\pm}$.
We have $d^++d^-=d$.
We can now state the integral period relations conjecture:

\begin{conj} \label{PRconj}   Assume $p$ is prime to the discriminant of $F$ and to $6N\varphi_F(N)$, then:
$$\Omega^\epsilon_{f_F}\sim (\Omega_f^+)^{d^+} (\Omega_f^-)^{d^-}.$$
\end{conj}
This conjecture follows easily from the general Deligne definition of periods if the relation is up to an algebraic number. 
We will prove some cases of this conjecture for totally real fields  which are Galois over $\Q$ of degree a power $2$ when $d^+=d^-$. 
The main arguments are similar to the case  $F$ is real quadratic over $\Q$. However, except in the quadratic case and signs $(+,-)$ or $(-,+)$, 
we do not know how to prove such a relation when $d^+\neq d^-$ even up to an algebraic integer.

\subsection{Main Results in the real quadratic case}\label{sectionRealQuadratic}
In this section, $F$ is a real quadratic field of discriminant $D$. We call $\alpha$ the corresponding quadratic Dirichlet character.
Let $p$ be a prime which does not divide $6ND\varphi_F(N)$ such that $p>k$. Let $f\in S_k(\Gamma_0(N))$ and $\ff=f_F$ 
its normalized base change as above. 
Considering the base change homomorphism $\T_F\rightarrow \T$, it gives rise following Section \ref{basechange} to a base change congruence ideal :
$$\eta^\sharp_f:=\eta^\sharp_{\lb_f}(\T_F)$$
Let $\H_{\m_F}^\epsilon=\H_!^\epsilon (F)_{\m_F}$. It is a $\T_F$-module and a finite $\cO$-module. Let
$$\eta^\sharp_f(\H_{\m_F}^\epsilon):=\eta^\sharp_{\lb_f}(\H_{\m_F}^\epsilon).$$
The following theorems prove an integral period relation together with a formula for these non base-change congruence ideals.
\begin{thm} \label{PRrealquadratic} Assume that (CV) holds, that
$p$ does not divide $6ND\varphi_F(N)$ and $p>k$. 
 Then, we have the integral period relation 
for any $\epsilon\in\widehat{W_F}$ such that $\epsilon(-1,-1)=-1$,
$$ \Omega_{f_F}^{\epsilon}\sim \Omega_{f}^{+}\Omega_{f}^{-}.$$
\end{thm}
\begin{proof}

This is a direct consequence of Propositions \ref{intperrel} and \ref{divisibgeneral} 
that are proved in section \ref{twistedL4realquadratic} below.
\end{proof}

Using this theorem one can prove 
\begin{thm}\label{congBCrealquadratic} Assume (CV) and that $p$ does not divide $6ND\varphi_F(N)$, $p>k$, 
then for any $\epsilon\in\widehat{W_F}$ such that $\epsilon(-1,-1)=-1$, we have
$$ \eta_f^\sharp(H^\epsilon_{\m_F}) \sim L^\ast(\Ad f \otimes \alpha,1):= {\Gamma(\Ad(f)\otimes\alpha,1) L(\Ad(f)\otimes\alpha,1)\over \Omega_{f}^{+}\Omega_{f}^{-}} $$
In particular, if $p$ divides $L^\ast(\Ad f \otimes \alpha,1)$, there is a congruence between $f_F$ and a non base-change Hilbert cusp eigenform. 
\end{thm}
\begin{proof}
The previous Theorem yields the divisibility 
\begin{equation}\label{div-Hida+goodPR} \eta_f^\sharp(H^\epsilon_{\m_F}) \;  | \;  {\Gamma(Ad(f)\otimes\alpha,1)L(Ad(f)\otimes\alpha,1). \over \Omega_f^+\Omega_f^- }
\end{equation}
It also yields
$$\eta_{f_F}(H^\epsilon_{\m_F})={ \Gamma(Ad(f_F),1)L(Ad(f_F),1)\over\Omega^\epsilon_{f_F}\Omega^{-\epsilon}_{f_F} }
\sim {\Gamma(\Ad(f),1) L(\Ad(f),1)\over \Omega_{f}^{+}\Omega_{f}^{-} }\cdot { \Gamma(\Ad(f)\otimes\alpha,1) L(\Ad(f)\otimes\alpha)\over \Omega_{f}^{+}\Omega_{f}^{-} }$$

On the other hand, consider the $\T_F$-module $M=H^\epsilon_{\m_F}$. We know that $M_K$ is free of rank one over $\T_F\otimes K$.
 Therefore,
if $M_{\T}$ denotes the largest submodule on which $\T_F$ acts through $\theta\colon\T_F\to \T$, we see that $M_{\T}\otimes K$ 
is free of rank one over $\T_K$. Thus, we see that $\eta_f(M_T)$ divides $\eta_f$ by Formula \ref{div-C0} of Section \ref{C0}. 
Applying Lemma \ref{BC-congruenceN} to $M=H^\epsilon_{\m_F}$, we obtain
$$\eta_{f_F}(H^\epsilon_{\m_F})\; | \;\eta_f\cdot \eta_f^\sharp(H^\epsilon_{\m_F})
$$
Therefore we have
$$ \eta_{f_F}(H^\epsilon_{\m_F})\sim {\Gamma(\Ad(f),1) L(\Ad(f),1)\over \Omega_{f}^{+}\Omega_{f}^{-} }\cdot { \Gamma(\Ad(f)\otimes\alpha,1) 
L(\Ad(f)\otimes\alpha)\over \Omega_{f}^{+}\Omega_{f}^{-} }\; | \;\eta_f\cdot \eta_f^\sharp(H^\epsilon_{\m_F})$$
On the other hand, from Theorem \ref{classical} we know that
$$\eta_f\sim {\Gamma(\Ad(f),1) L(\Ad(f),1)\over \Omega_{f}^{+}\Omega_{f}^{-}}$$
This yields the inverse divisibility of (\ref{div-Hida+goodPR}).
Thus these divisibilities are equivalences, hence our claim. Note that the proof shows in particular that the divisibility
$$\eta_{f_F}(H^\epsilon_{\m_F})\; | \;\eta_f\cdot \eta_f^\sharp(H^\epsilon_{\m_F})$$
is in fact an equality. 
\end{proof}

This has the following corollary:
 
\begin{cor}\label{congBCrealquadraticcor} Assume (CV) and that $p$ does not divide $6ND\varphi_F(N)$, $p>k$. 
Assume moreover that $(Irr_{F'})$ and that $\rho|_{G_F}$ is $N$-minimal.
Then, for any $\epsilon\in\widehat{W_F}$ such that $\epsilon(-1,-1)=-1$, we have
$$ \eta_f^\sharp \sim L^\ast(\Ad f \otimes \alpha,1):= {\Gamma(\Ad(f)\otimes\alpha,1) L(\Ad(f)\otimes\alpha,1)\over \Omega_{f}^{+}\Omega_{f}^{-}} $$

\end{cor}
\begin{proof}
If $H^\epsilon_{\m_F}$ is free over $\T_F$, we have $\eta_{f_F}=\eta_{f_F}(H^\epsilon_{\m_F})$; if moreover $\T_F$ is Gorenstein,
we have by Corollary \ref{Hida-factorization} that $\eta_{f_F}=\eta_f\cdot \eta_f^\sharp$. 
This implies that $\eta_f^\sharp=\eta_f^\sharp(H^\epsilon_{\m_F})$.
But under the assumption $(\Irr_F)$ and $(\Min_F)$, Taylor-Wiles method applied to $\T_F$ and $M=\H_{\m_F}^\epsilon$ shows that
$H^\epsilon_{\m_F}$ is free over $\T_F$, and $\T_F$ is complete intersection over $\cO$, hence is Gorenstein. The desired formula follows now
from the previous theorem.
\end{proof}

We now record some consequences towards the Bloch-Kato conjecture
for $\Ad(f)\otimes\alpha$. Let's recall the definition of the Selmer group first.
For a number field $E$ we put
$$\Sel_E( \Ad(f)\otimes\alpha):=\Ker\left(H^1(E,W_{f,\alpha})\rightarrow \bigoplus_v H^1(E_v,W_{f,\alpha})/L_v)\right)$$
where $W_{f,\alpha}=Ad(\rho_f)\otimes \cK/\cO(\alpha)$ and for each finite place $v$, $L_v$ is some local condition defined as follows. 
Let $E_v$ is the completion of $E$ at $v$ and $E_v^{un}$ its maximal unramified extension. 
Then  $L_v$ is defined as follows
$$L_v:=\left\{ \begin{array}{l@{\quad}l}
 H^1_{un}(E_v,W_{f,\alpha}) =Ker (H^1(E_v,W_{f,\alpha})\to H^1(E_v^{un},W_{f,\alpha}))& \hbox{if } v\not{|}p  \\
  H^1_{f}(E_v,W_{f,\alpha})   &  \hbox{if } v | p \\
\end{array}\right.
$$
where the $f$ stand  for the local finite condition at $v$ as defined by Bloch-Kato.
We refer to the beginning section 3 of \cite{BK90} for the precise definition.
\begin{cor}\label{BKreal} We assume (CV),  $(\Irr_F)$, $(\Min_F)$, $p>k$ and $p$ does not divide $6ND\varphi_F(N)$, then we have:
$$L^\ast(\Ad f \otimes \alpha)\sim\CH_\cO\,\Sel(\Ad f\otimes \alpha).$$
\end{cor}

\begin{proof}
Under our assumptions, it is well known that we have $R=T$ for $f$ and $R_F=T_F$ for $f_F$ where we have denoted by $R$ and $R_F$ 
the universal deformation rings with  Fontaine-Laffaille conditions at places dividing $p$ and minimal conditions at those dividing the level.
Moreover these rings are complete intersection. 
When $k=2$ and $F=\Q$ this is the original work of Wiles and Taylor-Wiles. In general, this follows from similar argument once an appropriate 
Taylor-Wiles system is constructed as in \cite{Fu06,Dim05}. We will see in Proposition \ref{freenessduality} and the proof of Theorem \ref{completeintersection} below that
these conditions are satisfied.
Since the rings are complete intersection, using Proposition \ref{LCI}, we can therefore rewrite the congruence ideals in terms of $\cO$-Fitting ideals of Selmer groups
$$\eta_f\sim \CH_\cO,\Sel(\Ad f)\quad\hbox{and}\quad  \eta_{f_F}=\CH_\cO\,\Sel(\Ad f_F).$$
Hence by Corollary \ref{Hida-factorization}
 and the decomposition
$$Sel_F(Ad(f))=Sel_\Q(Ad(f))\oplus Sel_\Q(Ad(f)\otimes\alpha),$$
 we get that $\eta_{f}^\sharp=\CH { \Sel(\Ad f\otimes \alpha)}$. So the results follows from the previous Theorem. \end{proof}

In particular for a semi-stable  elliptic curve, we have the following
\begin{cor} If $E/\Q$ is a semistable elliptic curve of conductor $N_E$ and $p$ a prime of good reduction not dividing $6_F(N_E)$
and for which the Galois representation $E[p]$ is irreducible over $F^\prime$. Assume that the conductor of $E[p]$ over $\Q$ is equal to $N_E$, 
then we have
$$\CH_\Zp Sel_\Q(Ad(E)\otimes\alpha) \sim { \Gamma(\Ad(E)\otimes\alpha)L^\ast(\Ad(E)\otimes\alpha)\over\Omega_E^+\Omega_E^-}$$
where $\Omega_E^+,\Omega_E^-$ are the N\'eron differentials associated to $E$.
\end{cor}
\begin{proof} By our hypothesis on $E$, the conditions of the previous corollary are satisfied for the weight $2$ cusp form $f$ of level $N_E$
 associated to $E$.
Our claim therefore  follows from the previous corollary after remarking that for a Weil curve $E$ and a prime $p>2$ such that $E[p]$
 is absolutely irreducible, the Manin periods $\Omega_f^{\pm}$ coincide with the Neron periods $\Omega_E^{\pm}$ up to a $p$-adic unit (see for instance \cite[Section3]{GV00}).
\end{proof}

The following subsection gives conditions for $\H^\epsilon_{\frak m}$ to be free over $\T_F$ and $\T_F$ is Gorenstein (which were used in the proof of
Corollary \ref{congBCrealquadraticcor}). The sections thereafter concern conjectures and partial results for arbitrary Galois totally real fields.
The real quadratic case, including the proofs of Propositions \ref{intperrel} and \ref{divisibgeneral} is treated in Section \ref{twistedL4realquadratic} below.

\subsection{Freeness of the cohomology and $R=T$ for a totally real field}\label{sectHilbcoh}
In this sectiont $F$ is a general totally real field like in section \ref{periodsF} from which we use the notations and assumptions.
We also assume that $p$ does not divide the discriminant $D$ of $F$. We first consder $\ff$ a general Hlbert modular form of parallel 
weight $k$ and $\m_F$ the corresponding maximal ideal of the Hecke algebra $h^F_k$. We assume throughout the section that $\m_F$ is not Eisenstein.




 Let $j\colon X_F\hookrightarrow X^\ast_F$ be the minimal compactification of $X_F$ and $\partial X_F^\ast=X_F^\ast-X_F$
its boundary; it is  $0$-dimensional. Let $\H_{\partial}^\bullet(X_F,\cV_k)=\H^\bullet(\partial X^\ast_F,i^\ast Rj_\ast\cV^F_k)$.
Let $\H^\bullet_{\m_F}$,
respectively  $\H_{c,\m_F}^\bullet$,  $\H_{!,\m_F}^\bullet$ and $\H_{\partial,\m_F}^\bullet$, 
the completion at $\m_F$ of $\H^\cdot(X_F,\cV_k)$, 
respectively $\H_c^\bullet(X_F,\cV_k)$, $\H_!^\bullet(X_F,\cV_k)$ and  $\H_{\partial}^\bullet(X_F,\cV_k)$.
Before stating the next proposition, let us introduce the Atkin Lehner involution
$$W_N=[U_0(\n)\left(\begin{array}{cc}0&-1\\N&0\end{array}\right)U_0(\n)]$$
Then we have:
\begin{pro}\label{freenessduality} We suppose that $p$ is prime to $6N\phi_F(N)$ and that $\m_F$ is not Eisenstein.
Then the following holds.
\begin{itemize}
\item[1)]We have $\H^\bullet_{\partial,\m_F}=0$,
%
\item[2)] We have $\H^\bullet_{\m_F}=\H_{!,\m_F}^\bullet=\H_{!,\m_F}^d=\H_{\m_F}^d$ 
and this $\cO$-module is torsion free. 
%
\item[3)] The Poincar\'e duality twisted by the Atkin-Lehner involution $W_N$
induces a perfect Hecke-bilinear pairing
$$[-, -]\colon \H_{\m_F}^d\times \H_{\m_F}^d\to \cO,$$
\end{itemize}
 \end{pro}
 %
\begin{proof}

For this we replace the $\Q$-algebraic group associated to $\Res_{F/\Q}\GL_2$ by a quasi-split group of unitary similitudes and apply Theorems 
by \cite{ACC+} and Caraiani-Scholze \cite{CS19}.
Let $E=F \Q(\sqrt{-1})$ and $J= \left(\begin{array}{cc} 0&1\\-1&0\end{array}\right)$; we consider the algebraic group 
$G$ defined over $F $ given, for any $F$-algebra $R$ by
$$G(R)=\{X\in\GL_2(E\otimes R); {}^t\overline{X}JX=\nu(X)\cdot J, \quad\nu(X)\in (F\otimes R)^\times\}$$ 
where the bar stands for the $F$ automorphism of $E$ induced by the complex conjugation on $\Q(\sqrt{-1})$.
The character  $\nu\colon G\to F^\times$ is the similitude factor and its derived subgroup
is $G^\prime \cong SL_2{}_{/F}$ as algebraic groups (see for example Bargmann, \cite[pp.589-591]{Ba47}). 
We have an exact sequence of algebraic groups over $F$:
$$1\to \Gm_{/F} \to \GL_2{}_{/F}\times Res^{E}_F(\Gm_{/\widetilde{F}})\to G\to 1$$
where the right hand side map is a central isogeny and is defined by $(g,x)\mapsto g.x1_2$. In particular, we see that
 the algebraic groups groups $\GL_2{/F}$, $H=\GL_2{}_{/F}\times Res^E_F(\Gm_{/E})$ and $G$ are isogeneous and have the same derived subgroup.
Therefore, the corresponding Shimura varieties and their compactifications have isomorphic connected components. More precisely, any connected component of
the Shimura variety for $GL_2{}_{/F}$ can be identified to a connected component of the Shimura variety for $G$.
In particular, it is sufficient to prove the proposition by replacing the cohomology 
of $X_F$ (or its boundary) by the cohomology of the Shimura variety $X^G_F$ for $G$ with level  the principal congruence subgroup $U_G(\n)$ in $G(\A_F)$.
We will denote again by $\cV_k(\cO)$ the local system on $X_F^G$ associated to an (non-unique) algebraic representation of $G$ that coincide with $\bigotimes_{\tau\in I_F}\cO[X_\tau,Y_\tau]_{k-2}$
on $\GL_2(F)$.
Since the cuspidal automorphic representation associated to $\ff$ has trivial central character, it induces canonically a (cuspidal) representation $\pi_G$
on $G(\bA_F)$ with trivial central character. Let $h_G$ be the Hecke algebra outside $\n$ acting faithfully on $\H^\bullet(X^G_F,\cV_k(\cO))$ and $\m_G$ be the maximal ideal of $h_G$
associated to $\pi_G$ modulo $p$. 
Let $j_G\colon X^G_F\hookrightarrow X^G_F{}^\ast$ be the minimal compactification of 
$X^G_F$ and  $i_G\colon \partial X^G_F=X^G_F - X^G_F{}^\ast \hookrightarrow X_F^G{}^\ast$ its boundary. 
We denote by  $\H_{\m_G}^\bullet$ and $\widetilde{\H}^\bullet_{\partial,\m_G}$  respectively the localizations of
$\H^\bullet(X^G_F,\cV_k(\cO))$ at $\m_G$ and  
$\H^\bullet(\partial X^G_F{}^\ast,i_G^\ast Rj_{G,\ast}\cV^F_k)$ at $\m_G$. We now go over each statement of the proposition.

1) By \cite[Theorem 2.4.2]{ACC+}, we have
$\widetilde{\H}^\bullet_{\partial,\m_G}=0$. Since the connected components of $X_F$ and $Y_F$ are isomorphic, and since 
we can descend from $Y_F$ to $X_F$ because $p$ is prime to $N\phi_F(N)$, we conclude that
$\H^\bullet_{\partial,\m_F}=0$. 
 
2) 
We have by \cite[Theorem 1.1]{CS19}
 $\H^i(X^G_F,\cV_k(\cO))=0$ for $i<d$ and $\H_c^i(X^G_F,\cV_k(\cO))=0$ for $i>d$. 
Therefore, $\H^i(X^G_F,\cV_k(\cO))_{\m_G}=0$ for $i\neq d$.  
Again, since the Shimura varieties $X^G_F$ and $X_F$ have the same connected components, 
and since we can descend from $Y_F$ to $X_F$ 
because $p$ is prime to $6ND\varphi_F(N)$, it follows that $\H^\bullet_{\m_F}=\H^\bullet_{c,\m_F}=\H_{!,\m_F}^\bullet=\H_{\m_F}^d$ 
and also that
 $\H^{d-1}(X_F,\cV^F_k(\cO/\varpi\cO))_{\m_F}=0$.
In particular, $\H_{\m_F}^d$ is torsion-free. 
 
3) After twisting by the Atkin-Lehner involution, 
the Poincar\'e duality pairing on
$$\H^d(X_F, \cV_k(\cO) )/(\cO -tors)\times \H^d_c(X_F,\cV_k(\cO)) /(\cO -tors) \longrightarrow \cO$$
 is still perfect and becomes Hecke-bilinear. Therefore, its restriction to
$\H^d(X_F, \cV_k(\cO) )_{\m_F}/(\cO -tors) \times \H^d_c(X_F, \cV_k(\cO))_{\m_F}/(\cO -tors)$ is still perfect.This proves our claim since by 1), we have
$\H^d_c(X_F,\cV_k(\cO))_{\m_F}=\H^d(X_F,\cV_k(\cO))_{\m_F}$ and these modules are $\cO$-torsion free by 2).
\end{proof}
 Thanks to the previous proposition, we are in position to construct a Taylor-Wiles system as in the work of Fujiwara \cite{Fu06} 
or Dimitrov \cite{Dim05}.
 Note that Dimitrov needed a big image assumption which is no longer needed to get the same conclusion as the proposition above.
\begin{thm}\label{completeintersection} Let $p>k$. 
Suppose that $(\Min_F)$  and $(\Irr_{F'})$ hold, then 
the module $\H^d_{\m_F}$ is free over the localized Hecke algebra $\T_F$, 
and this algebra is complete intersection (hence Gorenstein). Moreover, for every character $\epsilon$ of $\{\pm 1\}^d$,  
$\H^\epsilon_{\m_F}=\H^d_{\m_F}[\epsilon]$ is free of rank $1$ over $\T_F$.
\end{thm}

\begin{proof} 
We just give a sketch as the techniques to reach this result are well-known and follow the same lines
as the work of Dimitrov \cite{Dim05} or  Fujiwara \cite[Th.0.2]{Fu06} 
with the difference that Fujiwara considered a Taylor-Wiles system coming from quaternionic cohomology modules, 
associated either to a quaternionic Shimura curve or to a t
totally definite quaternion algebra.  Instead, we use as Dimitrov the Taylor-Wiles system coming from 
the middle degree cohomology of Hilbert-Blumental modular varieties.
 
We therefore  introduce as in \cite[section 11]{GT05} and \cite{Fu06,Dim05} the finite coverings $X_\Delta(Q)\to X_F$ 
corresponding to levels $U_0(\n)\cap U_\Delta(Q)\cap U_1(r)$, where $Q$ is a Taylor-Wiles set 
(and $r$ is a Wiles auxiliary prime which insures that the level group is neat but such that the
localisation at the maximal ideal $\m_Q$ won't be ramified at $r$, as in \cite[Lemma 11]{DT94} or \cite[Sections 7.4, 8.1]{Fu06}).
and
modules $\H_Q=\H^d(X_\Delta(Q),\cV_k(\cO))_{\m_Q}$'s with faithful action of localized Hecke algebras $\T_{F,Q}$ 
and diamond group algebras $\cO[\Delta_Q]$. 
They provide
a Taylor-Wiles system $(Q,\H_Q,R_{F,Q},\T_{F,Q})$ where
$\H_Q$ is a $\cO[\Delta_Q]$-finite free module such that 
$$\H_Q/I_Q\H_Q=\H_\emptyset=\H^d_{\m_F}$$ 
where $I_Q$ denotes the augmentation ideal of
$\cO[\Delta_Q]$. The  $\cO[\Delta_Q]$-freeness of $\H_Q$ and the descent formula follow directly from 
Statements 2) and 3) of the previous Proposition. 
The Taylor-Wiles machinery allows us to conclude that $\H^d_{\m_F}$ is free over $\T_F$, that 
the $N$-minimal $p$-Fontaine-Laffaille if $p\not\vert N$ and $k<p$)  universal deformation ring $R_F$ of 
$\overline{\rho}_f\vert_{\Gamma_F}$ is canonically isomorphic to $\T_F$ and that $\T_F$ is complete intersection.
\end{proof}

\subsection{Adjoint $L$ values for totally real fields and congruence ideal}\label{secttwspvalBK}

Let $\ff$ be as in the previous section and let $\lambda_\ff$ be the corresponding character  $\T_F\to\cO$.
Recall that $T_F$ acts faithfully on the finite free $\cO$-module $\H^\prime_{\m_F}=\H^d_{!,\m_F}/tors$ 
and that this module has a perfect  Hecke-equivariant pairing.

For any character $\epsilon\in\widehat{W}_F:=\Hom(W_F,\{\pm 1\})$, 
let $\H_{\m_F}^\epsilon$ be the direct factor of $\H^d_{\m_F}$
 given by the $\epsilon$-eigenspace. It is finite flat over $\cO$.
We apply the constructions of Section \ref{C0}
to the $T_F$-modules
\begin{itemize}
\item $M=\T_F$. We obtain the Hecke congruence module
$C_0(\ff):= C_0^{\lb_\ff}(\T_F)=  (\cO\times \T^\prime_F)/\T_F=\cO\otimes_{\T_F} \T^{\cO}_F$
and the Hecke congruence ideal  $\eta_{\ff}=\eta_{\lb_\ff}(\T_F)$ 
given by $\lb_\ff(\Ann_{\T_F}(\ker\lb_\ff))$ viewed as an ideal of $\cO$. 

\item $M=\H_{\m_F}^\epsilon$. We obtain
the cohomologial congruence module $C_0^{\lb_\ff}(\H_{\m_F}^\epsilon)$ and the 
cohomological congruence ideal $\eta_\ff^\epsilon:=\eta_{\lb_\ff}(\H_{\m_F}^\epsilon)$. 
Since $\H_{\m_F}^\epsilon\otimes K$ is free of rank one over $\T_F\otimes K$,
we know by ({div-C0}) that
$$C_0^{\lb_\ff}(\H_{\m_F}^\epsilon)=\cO/\eta_\ff^\epsilon.$$
\end{itemize}

Let $-\epsilon$ be the character which on the canonical basis 
$(w_\tau)_\tau$ of $\{\pm 1\}^{I_F}$ is given by
$(-\epsilon) (w_\tau)=- (\epsilon(w_\tau))$. Because of the perfect duality between $\H_{\m_F}^\epsilon$ and $\H_{\m_F}^{-\epsilon}$, 
by Lemma \ref{Pontryaginprop}, we have as $\cO$-modules
$$C_0(\H_{\m_F}^\epsilon)=C_0(\H_{\m_F}^{-\epsilon})=\cO/[\delta_{\ff}^\epsilon,\delta_{\ff}^{-\epsilon}]$$
so we have
$$\eta_{\ff}^\epsilon=[\delta_{\ff}^\epsilon,\delta_{\ff}^{-\epsilon}].$$
Of course, if $\H^\epsilon_{\m_F}$ is free over $\T_F$, we have that $\eta_\ff=\eta_\ff^\epsilon$. T
he following result is a generalization of
Theorem \ref{classical}.

\begin{pro} \label{Lvalcong} Assume that $p$ is prime to $6ND\phi_F(N)$ and $p>k$. We have for any $\epsilon\in \widehat{W}_F$,
$$\eta_\ff^\epsilon\sim{\Gamma(\Ad \ff ,1)L(\Ad \ff,1)\over \Omega_{\ff}^{\epsilon}\Omega_{\ff}^{-\epsilon}} $$
If moreover the hypothesis of Theorem \ref{completeintersection} hold, then 
 $$\eta_\ff \sim {\Gamma(\Ad \ff ,1)L(\Ad \ff ,1)\over \Omega_{\ff}^{\epsilon}\Omega_{\ff}^{-\epsilon}}$$
 \end{pro}
\begin{proof}  The proof follows the same lines as the case $F=\Q$ and follows from the computation of $[\delta_{\ff}^\epsilon,\delta_{\ff}^{-\epsilon}]$ and the Poincar\'e duality established in Proposition
\ref{freenessduality}. This result was also proved by Dimitrov in \cite[Theorem 4.4, Theorem 4.11, Formula (19)]{Dim05}. 
When the conditions of Theorem \ref{completeintersection} are satisfied, $\H^\epsilon_{\m_F}$ is free over $\T_F$
and therefore $\eta_\ff=\eta_\ff^\epsilon$ which implies the second claim.
\end{proof}

We will apply this result when $\ff=f_F$ is the base change of our original modular form $f$.

\subsection{Twisted adjoint $L$ values and congruences for base change} \label{twistedLvalueANDcong}
From now on, we again assume that the totally real field $F$ is abelian over $\Q$ with Galois group $G$ and recall that we denote by $f_F$ the base change of $f$ to $F$.
Let $I_G$ be the augmentation ideal of $\Z[G]$ that we look at as a representation of $G$ of degree $d-1$. We have therefore a factorization of $L$-functions:
$$L(\Ad f_F ,1)=L(\Ad f ,1)L(\Ad f\otimes I_G ,1).$$
Let us put
$$L^\ast(\Ad\otimes I_G,1)={\Gamma(\Ad f,1)^{d-1}L(\Ad f\otimes I_G ,1)\over (\Omega_{f}^{+}\Omega_{f}^{-})^{d-1}}.$$
where $D$ is the discriminant of $F$. 
The following statement seems well known.
\begin{pro} \label{divi-Hidal}    Assume that $F$ is abelian over $\Q$ and that $p>k-2$ and prime to $2N\phi_F(N)$, then
$$L^\ast(\Ad\otimes I_G,1)\in \cO.$$
\end{pro}
\begin{proof}This is due to combined arguments and computations of Sturm and Hida.
If if $F$ is real quadratic this is due to Hida \cite{H99}. In general, for each non trivial character of $\alpha$ of $G$, it follows from the integral formula for this twisted $L$-value due at $s=0$ to Sturm 
(see for example \cite[Lemma 4.5]{H90}) and using the functional equation to get the result at $s=1$ hat 
$$L^\ast(\Ad f,\alpha,1 ):=G(\overline{\alpha})^2{\Gamma(\Ad f\otimes\alpha ,1)L(\Ad f\otimes\alpha ,1)\over \Omega_{f}^{+}\Omega_{f}^{-}}\in \cO$$ 
One can also use the doubling method and the formula follows from \cite{BS00} in that case.
By using the conductor-discriminant formula $\sqrt{D}=\prod_{\alpha\in\widehat{G}}G(\alpha)$ (up to a unit), the result follows.
We are not giving more details as this fact won't be used in the rest of this paper when $F$ is not real quadratic.
\end{proof}
Similarly, we write
$$ L^\epsilon(\Ad f_F,1):={\Gamma(\Ad f_F ,1)L(\Ad f_F ,1)\over \Omega_{f_F}^{\epsilon}\Omega_{f_F}^{-\epsilon}}$$
Recall that these numbers calculate the cohomological congruence ideals of $f$ and $f_F$ (see Proposition \ref{Lvalcong}).
Let us notice the following lemma.
\begin{lem}\label{basechangehom} If $p$ does not divide the degree $d=[F:\Q]$, the homomorphism $\theta\colon\T_F\to \T$ is surjective.
\end{lem}

\begin{proof}
%
 Let $\ell$ be a rational prime,  prime to $ND$. If $\ell=\l_1\ldots\l_r$ decomposes in $F$ in primes of degree $f$,
Let $t_\ell$, resp. $t_{\l_i}$, be the Hecke parameter of $f$, resp. $f_F$, at $\ell$, resp. $\l_i$. Thus, $\Tr(t_\ell)=T_\ell$ and $\Tr(t_{\l_i})=T^F_{\l_i}$. 
Therefore, $\theta(T^F_{\l_i})=Q_f(T_\ell,S_\ell)$ where $Q_f(T,S)=T^f+\ldots \in\Z[T,S]$ is a monic polynomial in $T$ defined by the relation
$X_1^f+X_2^f=Q_f(T,S)$ for $T=X_1+X_2$ and $S=X_1X_2$. Note that ${\partial Q_f(T,S)\over\partial T}=f\cdot Q_{f-1}(T,S)$.
Note that in any ring $B$ in which $f$ is invertible and for any $s\in B-\{0\}$, $Q_f(T,s)$ has no multiple root.

Let $\T^\prime=\theta(\T_F)\subset \T$ ; it contains the $Q_f(T_\ell,S_\ell)$ for all $\ell$'s prime to $ND$. 
The equation $Q_f(X,S_\ell)-\theta(T^F_{\l_i})=0$ has coefficients in the complete
local ring $\T^\prime$.
it has a solution modulo its maximal ideal because the residue field $\F$ of $\T^\prime$ is the same as that of $\T$.
 Since $p$ does not divide $f$, it has no multiple root in $\F$.
Therefore, by Hensel's lemma, it has a solution in $\T^\prime$. 
Hence, by Chebotarev density theorem, we conclude that
$\T^\prime=\T$.
\end{proof}

Recall that if the rings $\T$ and $\T_F$ are Gorenstein and $p$ does not divide $d$, it therefore
follows from Corollary \ref{Hida-factorization} that there is a decomposition of the congruence ideals
\begin{equation} \eta_{f_F}=\eta_f\eta_f^\sharp \label{factorization-general-F}
\end{equation}
with $\eta_f^\sharp=\eta_{\lb_f}^\sharp(\T_F)$ as  defined in Section \ref{basechange}.
On the other hand, it follows from our conjecture \ref{PRconj}, as it was  conjectured by Hida in \cite{H94} in the real quadratic case, that
we must have the integral period relation
\begin{equation} \label{intperrel2}
\Omega_{f_F}^{\epsilon}\Omega_{f_F}^{-\epsilon}\sim  (\Omega_{f}^{+}\Omega_{f}^{-})^d.
\end{equation}
The relation \eqref{intperrel2} implies immediately
that $$L^\epsilon(\Ad f_F,1) \sim L^\ast(\Ad f ,1)L^\ast(\Ad f \otimes I_G,1)$$
This together with \eqref{factorization-general-F} suggest the following conjecture that will be proved in the real quadratic case.

\begin{conj} \label{twistL_NBC} Assume $p$ does not divide $6dDN\phi_F(N)$ and that $(\Irr_F)$ holds. 
We further suppose that $\T$ and $\T_F$ are Gorenstein,
then
$$\eta_f^\sharp\sim L^\ast(\Ad f \otimes I_G) $$
In particular, if $p$ divides $L^\ast(\Ad f \otimes I_G)$, there is a congruence between $f_F$ and a non base change Hilbert cuspform.
\end{conj}


With a bit more assumptions, we can prove the following

\begin{pro} \label{eqBK} Assume the conditions $(\Irr_{F^\prime})$, $(\Min_F)$, $p>k$ and $p$ prime to $dDN\phi_F(N)$. 
Then if  Conjecture \ref{twistL_NBC} holds, we have
$$L^\ast(\Ad\,f\otimes I_G,1)\sim  \CH\Sel_\Q(\Ad(\rho_f)\otimes I_G)$$
\end{pro}

\begin{proof} Under these conditions, it is well-known that

1) The deformation problem of $\overline{\rho}_f\vert_{\Gamma_F}$ for deformations preserving $N$-minimality  
and the Fontaine-Laffaille condition with Hodge-Tate weights $0$ and $k-1$ is representable. Let $R_F$ be the corresponding universal
 $\cO$-algebra,

2) There is an $\cO$-algebra isomorphism  $R_F\to \T_F$ (it comes from the proof of Theorem \ref{completeintersection}). 

Passing to the the relative differentials, we have an isomorphism
$$C_1^{\lb_{f_F}}(R_F)=\Omega_{R_F/\cO}\otimes_{\lambda_f} \cO \cong \Omega_{\T_F/\cO}\otimes_{\lambda_f} \cO=C_1^{\lb_{f_F}}(T_F)$$
By an argument due to Mazur, the left-hand side is isomorphic to the Pontryagin dual of $\Sel( \Ad\,\rho_{f_F})$. 
On the other hand, since we know that $T_F$ is complete intersection, by Wiles criterion, 
$C_1^{\lb_{f_F}}(T_F$ and $C_0^{\lb_{f_F}}(T_F$ have the same Fitting ideal and therefore

\begin{equation}\eta_{f_F}\sim \CH \Omega_{R_F/\cO}\otimes_{\lambda_{f_F}} \cO =\CH \Sel_F( \Ad\rho_{f}) \label{KC}
\end{equation}
\noindent
Since $p>d$, we have also the decomposition
\begin{equation} \label{fact-Selmer}\Sel_F( \Ad\rho_{f})=Sel_\Q(\Ad\rho_f\otimes_\cO\cO[G])=\Sel_\Q( \Ad\rho_{f})\oplus\Sel_\Q( \Ad\,\rho_{f}\otimes I_G)
\end{equation}
and by Wiles criterion again, we have $\eta_f=\CH \Sel_\Q( \Ad\rho_{f})$ since $\T$ is complete intersection.
Now from \eqref{factorization-general-F} , \eqref{KC} and the translation of \eqref{fact-Selmer} as a factorization of Fitting ideal, we obtain
$$\eta_{f}^{\sharp}\sim  \CH\Sel_\Q(\Ad(\rho_f)\otimes I_G)
$$
We conclude the proof by using Conjecture \ref{twistL_NBC}. 

\end{proof}


\subsection{Twisted adjoint $L$ values and congruence criterion for a real quadratic field}\label{twistedL4realquadratic}
In this section $F$ is real quadratic of discriminant $D$ of associated quadratic Dirichlet character $\alpha$ as in section \ref{sectionRealQuadratic}. 
In  \cite[Theorem 5.2]{H99}, Hida proved that if there is a congruence bewteen $f_F$ 
and a non base change modulo $\p$, then $\p$ divides $L^\ast(\Ad f \otimes \alpha)$. The following proposition is a strengthening of his result.
Note that $\epsilon(-1,-1)=-1$ implies $-\epsilon(-1,-1)=-1$ by definition of $-\epsilon$.

\begin{pro} \label{Hida-linearform} For $F$ real quadratic, 
$p$ does not divide $DN\phi_F(N)$, $p>k-2$
then for any $\epsilon\in\widehat{W}_F$ such that $\epsilon(-1,-1)=-1$, 
$$\eta_{f_F}^\sharp(\H^\epsilon_{\m_F})\; |\; L^{\epsilon}(\Ad f\otimes \alpha):=
{\Gamma(\Ad f\otimes \alpha,1)L(\Ad f\otimes \alpha,1)\over \Omega_{f_F}^{\epsilon}}.$$
\end{pro}
\begin{proof} Following Hida  \cite[Section 4]{H99}, we now define an integral linear form 
$$\Ev\colon \H^2_!(X_F,\cV_k^F(\cO))_{\m_F}\to \cO.$$
so that we can put ourself in the situation of Proposition \ref{linear-BC}. 
It will be obtained by pairing with a twisted $2$-cycle given by a
modular curve inside the Hilbert modular surface $X_F$. More precisely, let us fix $\varphi$ be a finite order character
of $F$ of level $\mathfrak d$ the different of $F/\Q$ and whose restriction to $(\Z/D\Z)^\times$ is given by $\alpha$.  
For any open compact $U\subset G(\A_F^\infty)$ let us denote by $X_U$ the corresponding 
Hilbert modular surface of level $U$. Then for
any $g\in GL_2(\bA_F^\infty)$ with $g_p=1$, there is a canonical map 
$$r_g\colon H^2_c  (X_U,\cV_k^F(\cO))\rightarrow H^2_c(X_{g^{-1}Ug},\cV_k^F(\cO))$$
induced by $X_{g^{-1}Ug}\to X_U$ obtained by right translation by $g$.
Then
we consider the twisting operator
$$R(\varphi)\colon H^2_c(X_F,\cV_k^F(\cO))\longrightarrow H^2_c(X_F^\prime,\cV_k^F(\cO))$$
where 
$X'_F$ is the Hilbert modular surface of level $U_1(\n\delta^2)$ and where
$R(\varphi)$ is defined by
$R(\varphi):=\sum_{u\in R}\varphi(u)\cdot r_{\alpha_u}$
with $\alpha_u=\left(\begin{array}{cc} 1 &u\\0&1 \end{array}\right)$ and where $R$ is a system of representative elements in $\bA_F^\infty$ of $(\mathfrak d^{-1}/O_F)^\times$ where the latter
stands for the subset of elements whose annihilator in $O_F$ is exactly $\mathfrak d$ (see \cite[(6.9)]{H94}). Now let $X'_\Q$ be the modular curve of level $U_1^\Q(ND^2)$, we then consider the
restriction map to $X'_\Q$  followed by the projection from the restricted local system onto $\cO$ as follows:
$$Res\colon H^2_c(X'_F,\cV_k^F(\cO))\to H^2_c(X'_\Q,\iota^*\cV_k^F(\cO))\longrightarrow H^2_c(X'_\Q,\cO)\cong \cO$$
where $\iota$  is the closed embedding  $X'_\Q\hookrightarrow X'_F$, and the second arrow is induced by the isomorphism 
$\iota^* \cV_k^F(\cO)\cong \cV_k(\cO)\otimes\cV_k(\cO)$ and the projection of $\GL_2(\Zp)$-representations
$V_k(\cO)\otimes V_k(\cO)\to\cO$ induced by the Clebsh-Gordan decomposition. That projection can be described explicitly (see the beginning
 of section 3 of \cite{H99}) and one sees it respects integrality thanks to the condition $p>k-2$.
 Finally the last isomorphism is obtained by integrating on the modular curve $X'_\Q$. 
Since we know thanks to Proposition \ref{freenessduality}
  that $H^2_!(X'_F,\cV_k^F(\cO))_{\m_F}=H^2_c(X'_F,\cV_k^F(\cO))_{\m_F}$, the map $Ev$ can be obtained after localization at $\m_F$ as 
$$Ev:= (Res\circ R(\varphi))_{\m_F}.$$
This map can be defined similarly by replacing the cohomology of the integral local system $\cV_k^F(\cO)$ by the de Rham cohomology
of $\cV_k^F(\C)$ and was computed explicitly by Hida \cite[Prop 4.1]{H99}. 
In particular on $\omega^\epsilon(\mathbf f)$, he obtains the following:
$$Ev(\omega^\epsilon(\mathbf f) )=  \left\{ \begin{array}{l@{\quad}l}
 0 & \hbox{if } \mathbf f \hbox{ does not come from base change}  \\
A\cdot  \Gamma(\Ad f\otimes \alpha,1)L(\Ad f\otimes \alpha,1)  &  \hbox{if } \mathbf f =f_F\\
\end{array}\right.
$$
where $A$ is an explicit constant which is prime to $p$ by our hypothesis. We deduce that $Ev(\delta^\epsilon_{f_F})\sim L^{\epsilon}(\Ad f\otimes \alpha)$.
Then our claim now follows from Proposition \ref{linear-BC}  with $M=\H^\epsilon_{\m_F}$ and $\Phi=Ev$.\
\end{proof}


%
We can then deduce the following
\begin{pro} \label{intperrel} We keep the same hypothesis as Proposition \ref{Hida-linearform} and Theorem \ref{classical}.b.  Then,
%
for any $\epsilon\in\widehat{W}_F$ such that $\epsilon(-1,-1)=-1$, we have the period relation:
$${  \Omega_{f_F}^{-\epsilon}\over \Omega_{f}^{+}\Omega_{f}^{-}}\in\cO$$
\end{pro}
\begin{proof}
From Lemma \ref{BC-congruenceN}. we have
$$\eta_{f_F'}(\H^\epsilon_{\m_F})=\eta_f((\H^\epsilon_{\m_F})_T).\eta^\sharp_f(\H^\epsilon_{\m_F})$$
Since $\eta_f((\H^\epsilon_{\m_F})_T)$ devides $\eta_f$, using the relation
$$\eta_{f_F'}(\H^\epsilon_{\m_F})\sim  {L(\Ad f_F\otimes \alpha,1) \over\pi^{2k+2} \Omega_{f_F}^{\epsilon}\Omega_{f_F}^{-\epsilon}}$$
and 
$$ \eta_f\sim {L(\Ad\,f,1)\over\pi^{k+1}\Omega_f^+\Omega_f^-}$$
by Theorem \ref{classical}, we have the divisibilities 
$${ L(\Ad f_F\otimes \alpha,1) \over \pi^{2k+2}\Omega_{f_F}^{\epsilon}\Omega_{f_F}^{-\epsilon}}\; |\; {L(\Ad\,f,1)\over\pi^{k+1}\Omega_f^+\Omega_f^-}\times \eta^\sharp_f(\H^\epsilon_{\m_F})\; |\; 
{L(\Ad\,f,1)\over\pi^{k+1}\Omega_f^+\Omega_f^-}  \times {L(\Ad f\otimes \alpha,1)\over \pi^{k+1} \Omega_{f_F}^{\epsilon}}.$$
where the last divisibility follows from Proposition \ref{Hida-linearform}. Since the quotient of the right hand side by the left 
hand side is ${  \Omega_{f_F}^{-\epsilon}\over \Omega_{f}^{+}\Omega_{f}^{-}}$, our claim follows.
\end{proof}

\begin{rem} 
We do not need to have the free-ness of the cohomology of the Hilbert modular surface to get this periods relation.
\end{rem}

We shall prove a weaker form of the other divisibility under the assumptions of a theorem by Cornut-Vatsal (see \cite{C02} and \cite{Va02} for the weight $k=2$, and Chida-Hsieh \cite{CH16} for any $k\in[2,p+1]$), 
 the following integral period relation holds

\begin{pro} \label{divisibgeneral} Assume $p>k$, $p$ is prime to $6.N\sharp (\Z/N\Z)^\times$,  and that (CV) and $(\Irr_{F'})$ hold. Then for any $\epsilon\in\widehat{W}_F$, 
we have 
$${(\Omega_{f}^{+}\Omega_{f}^{-})^2\over \Omega_{f_F}^{\epsilon}\Omega_{f_F}^{-\epsilon}}\in \overline{\Z}_p.$$
\end{pro}

\begin{proof} The form has level $\Gamma_0(N)$ with trivial Nebentypus) and $N$. Let $q$ be the prime for which (CV) holds.
We choose an imaginary quadratic field $K$ in which all the prime factors of $N$ other than $q$ split while $q$ remains inert
and $p$ splits in $K$.
We also fix a prime $\ell$ prime to $Np$ as in Cornut-Vatsal (see \cite[Theorem]{C02} and \cite[Theorem 1.2]{Va02})  and Chida-Hsieh \cite[Theorem C]{CH16}. In this last reference, the authors are using a different period they denote $\Omega_{f,N^-}$ which may be different, \` a priori, from Hida's period $\pi^k \Omega_f^+\Omega_f^-$. However, they prove
under their assumption $\mathbf (CR^+)$ that both periods differ by a $p$-adic unit by \cite[Prop. 6.1]{CH16}.  We may therefore use their result since their hypothesis $(CR^+)$ is implied by ours.

Then there exists an anticyclotomic character $\xi$ of conductor a power of $\ell$ such that for any $\chi\in\{1,\alpha\}$ where $\alpha$ stands for the quadratic character associated to $F/\Q$, we have
$${L_K(f\otimes\chi,\xi,k/2)\over \pi^k \Omega_{f}^+\Omega_{f}^{-}} \in \overline{\Z}_p^\times$$
 
On the other hand, by Hida (Rankin-Selberg method) ${L_{KF}(f_F,\xi,1)\over\pi^k\Omega_{f_F}^+\Omega_{f_F}^{-}}\in\overline{\Z}_p$.
Indeed,
$${L_{KF}(f_F,\xi,1)\over\pi^{2k}\Omega_{f_F}^+\Omega_{f_F}^-}={<f_F,\theta(\xi)E>\over <f_F,f_F>}\eta^{coh,\epsilon}_f=\eta_f\cdot a(1,1_f(\theta(\xi)E))$$
where $E$ and $\theta$ are respectively some Einsenstein series and Theta function with integral Fourier coefficients.
Here we used that $\eta_f^{coh}=\eta_f$ under our assumptions (Irr) and $p>k$.
Hence, the quotient of 
${L_{KF}(f_F,\xi,1)\over\pi^{2k}\Omega_{f_F}^+\Omega_{f_F}^{-}}$
by
$\prod_{\chi\in\{1,\alpha\}}{L_{K}(f\otimes\chi,\xi,1)\over\pi^k\Omega_{f}^+\Omega_{f}^{-}}$
is a $p$-adic integer.
Our claim follows since this quotient is equal to
${(\Omega_{f}^+\Omega_{f}^{-})^2\over\Omega_{f_F}^+\Omega_{f_F}^{-}}.$
\end{proof}

\subsection{Case of totally real fields of $2$-power degree over $\Q$}

Let $F$ be a Galois totally real field with Galois group $G$ of $2$-primary order.
In that case we need a generalization of Proposition \ref{intperrel} for a relative real quadratic extension $F/F_1$.
Actually, the Hida evaluation linear form is also available in this context:

$$ev_{F/F_1}\colon \H^d(Y_F,\cV^F_k(\cO))\to\cO$$
it is given by the cap pairing with the Hilbert variety for $F_1$ (followed by the evaluation on the lowest weight vector) $\omega\mapsto\omega\cap[X_{F_1}]$.
Note that $[X_{F_1}]$ is in $\H_{2d_1}(X_F)$ but $2d_1=d$. 

We use this time \cite[Theorem 6.1]{H99}. With this, we can prove the following generalization of Proposition \ref{intperrel}.
Let us consider the embedding $i\colon W_{F_1}\to W_F$ where, for $e=(e_{\tau_1})\in W_{F_1}$,
 $i(e)$ is defined as follows: for each embedding $\tau_1$ of $F_1$,
and for $\tau$, $\tau^\prime$ the two embeddings above $\tau_1$, one puts $i(e)_\tau=i(e)_{\tau^\prime}=e_{\tau_1}$.
Before stating our next lemma, we need a definition.
\begin{de}
We say that $\epsilon\in \widehat{W_F}$ is balanced if and only if  $d^+_\epsilon=d^-_\epsilon$. 
\end{de}
Notice that $\widehat{W_F}$ containes balanced character if and only if $d=[F:\Q]$ is even.
\begin{lem} \label{intperrelrelative} For $F$ totally real and $F/F_1$ quadratic, assuming that 
$p$ does not divide $N\phi_F(N)$, $k<p$, then for any balanced $\epsilon\in\widehat{W_F}$  such that 
$\eta_{f_{F_1}}=(\frac{L(\Ad f_{F_1}),1)}{ \pi^{(k+1)d/2}\Omega_{f_{F_1}}^{-\epsilon_1} \Omega_{f_F}^{\epsilon_1}})$ for $\epsilon_1=i^\ast(\epsilon)\in \widehat{W_{F_1}}$,
%
we have the period relation
$${ \Omega_{f_F}^{\epsilon}\over \Omega_{f_{F_1}}^{\epsilon_1}\Omega_{f_{F_1}}^{-\epsilon_1}}\in\cO.$$
\end{lem} 
\proof We first apply Corollary \ref{Hida-factorization} to the map $\T_F\rightarrow \T_{F_1}$ and $\lb_1 \colon \T_{F_1}\rightarrow\cO$ (resp $\lb  \colon \T_{F}\rightarrow\cO$ ) is 
the character of the Hecke algebra attached to the eigenform  $f_{F_1}$ (resp.  to the eigenform $f_F$).
Let us write $\eta_{f_F}^{\epsilon}$ for $\eta_\lb((\H^d(Y_F,\cV^F_k(\cO))^\epsilon ))$ and
$\eta^{\sharp,\epsilon}_{f_{F_1}}$ for $\eta^\sharp_{\lb_1}(\H^d(Y_F,\cV^F_k(\cO))^\epsilon )$.
Then, we have the divisibility
$$\eta_{f_F}^\epsilon \supset \eta_{\lb_1} (\H^d(Y_F,\cV^F_k(\cO))^\epsilon_{\T_{F_1}}  ) \cdot \eta^{\epsilon,\sharp}_{f_{F_1}}     \supset \eta_{f_{F_1}} \cdot \eta^{\epsilon,\sharp}_{f_{F_1}} $$
where the second inclusion is obtained using \eqref{div-C0}.
Now from the Proposition \ref{linear-BC} applied to $ev_{F/F_1}$, we obtain
$$\frac{L(\Ad(f_{F_1})\otimes\alpha,1)}{\pi^{(k+1)d/2}\Omega_{f_F}^{\epsilon} } \in\eta^\sharp_{f_{F_1}} $$
On the other hand, we know that $\eta_{f_F}^\epsilon$ is generated by
$$\frac{L(\Ad(f_{F}),1)}{\pi^{(k+1)d} \Omega_{f_F}^{-\epsilon} \Omega_{f_F}^{\epsilon}}=
 \frac{L(\Ad(f_{F_1}),1)}{\pi^{(k+1)d/2}\Omega_{f_F}^{-\epsilon} }\cdot \frac{L(\Ad(f_{F_1})\otimes\alpha,1)}{\pi^{(k+1)d/2}\Omega_{f_F}^{\epsilon} }$$
Combining our assumption on  $\eta_{f_{F_1}}$ and the previous divisibilities implies our claim.
\qed
\medskip

To obtain the opposite divisibility we proceed like in Proposition \eqref{divisibgeneral}  for  the real quratic case.
\begin{lem} Assume that $p>k$ is prime to $N\phi_F(N)D$ and that (CV) and $(\Min_F)$ hold, then we have
$${ ( \Omega_{f_{F_1}}^{\epsilon_1}\Omega_{f_{F_1}}^{-\epsilon_1})^2    \over    \Omega_{f_F}^{\epsilon} \Omega_{f_F}^{-\epsilon} }\in\cO.$$
\
\end{lem}
\begin{proof}
The same argument as in the proof of Proposition \eqref{divisibgeneral} applies by replacing the theorem of Cornut-Vatsal generalized by Chida-Hiseh by a further
 generalization due to Pin-Chin Hung \cite{HPC} to a totally real ground field instead of $\Q$.

\end{proof}

We are now ready to prove the following 

\begin{thm} Let $F$ be a totally real field Galois of $2$-power degree $d$ over $\Q$.  Assume $p$ does not divide $6N\varphi_F(N)$ and that  $p-1>d(k-1)$. Further assume that 
 $CV$ and $(\Min_F)$ hold. Then, for any balanced $\epsilon\in \widehat{W_F}$ , we have
$$\Omega^\epsilon_{f_F}\sim (\Omega_f^+\Omega_f^-)^{d/2}$$
In other words, Conjecture \ref{PRconj} holds for such $\epsilon$ if $[F:\Q]$ is a power of $2$.
\end{thm}
\begin{proof}
We proceed by induction on the degree of $F$. Notice first that the starting point of the induction is Theorem \ref{PRrealquadratic}, so we
may assume that $[F:\Q]\geq 4$. We choose $F_1\subset F$ like in the two previous lemmas and such that $\epsilon_1$ is balanced.
Assuming the result for $F_1$ will therefore insure the result for $F$, once we have checked that the assumptions of the Lemmas above are satisfied.
This fact follows from our hypothesis and Theorem \ref{completeintersection} applied to $F_1$.

\end{proof}

\section{The imaginary quadratic case}

\subsection{Differential forms and periods}\label{periods-imagquad}

Let $F$ be an imaginary quadratic field of discriminant $-D$ prime to $Np$. We denote by $\frak d\subset O_F$ its different ideal and by $\alpha=({-D\over\cdot})$ the associated quadratic character. 
We now consider the space  $S^F_{k,k}(N;\C)$ of cuspidal Bianchi modular forms of parallel weight $k$ and level $N$. Let us recall that
an element $\mathbf f\in S^F_{k,k}(N;\bC)$ is a function on $\GL_2(F)\backslash \GL_2(\bA_F)$ taking values in $L(2n+2,\C):=\C[S,T]_{2n+2}$ that we identify with the
 subspace of polynomials in the two variables $(S,T)$which are homogenous of degree  $2n+2$ with $n=k-2$, 
$$g\mapsto \mathbf f(g)=\mathbf f(g;S,T)=\sum_{i=0}^{2n+2}\mathbf f_i(g)S^{2n+2-i}T^i\in L(2n+2,\C)$$
satisfying the following conditions:
\begin{itemize}
\item[(i)] $D_1.\mathbf f=D_c.\mathbf f=(n+\frac{n^2}{2}).\mathbf f$ where $D_1$ and $D_c$ are the two Casimir operators  attached to the two embeddings of $F$ into $\C$,
\item[(ii)] $\mathbf f(gk_\infty;(S,T))=f(g,(S,T)k_\infty^{-1})$ for any $k_\infty\in U_2(\C)$,
\item[(iii)] $\mathbf f (gz_\infty)=|z_\infty|_\C^{-n}\mathbf f(g)$ for any $z_\infty\in Z(\C)\cong\C^\times$,
\item[(iv)] $\mathbf f(gk_f)=\mathbf f(g)$ for all $k_f\in U_0^F(N)$.
\item[(v)] $\mathbf f$ is cuspidal. That is
$$\int_{\bA_F/F}\mathbf f(\begin{pmatrix}1&x\\0&1\end{pmatrix}g)dx=0 \hbox{ for all }g\in \GL_2(\bA_F).$$
\end{itemize}

We will say that $\mathbf f $ is a primitive eigenform if the span of the translates under the action of $\GL_2(\bA_F)$ of the $\mathbf f_i$ 
is an irreducible cuspidal representation that we will denote $\pi(\mathbf f)$. Any $\mathbf f\in S_{k,k}^F(N,\C)$ has a Fourier expansion of the form
$${\mathbf f}(\begin{pmatrix}y&x\\0&1\end{pmatrix})=|y|_F\sum_{\xi\in F^\times}a(\xi y_f\frak d,\mathbf f)W_n(\xi y_\infty)\mathbf e_F(\xi x) $$
(cf \cite[(6.1)]{H94}) where for $y_\infty\in\bR^\times $,
$$W_n(y_\infty)=\sum_{j=0}^{2n+2} \left(2n+2\atop j\right) (\frac{y_\infty}{\sqrt{-1}|y_\infty|})^{n+1-j}K_{j-n-1}(4\pi|y_\infty|)S^{2n+2-j}T^j$$
with $K_\alpha$ the 
modified\footnote{$K_\alpha$ satisfied the differential equation $\frac{d^2K_\alpha}{dx^2}+\frac{dK_\alpha}{dx}-(1+\frac{\alpha^2}{x^2})K_\alpha=0$ and $K_\alpha (x)\sim \sqrt{\frac{\pi}{2x}}e^{-x}$ as $x\rightarrow\infty$.} 
Bessel function of order $\alpha$. Here $\xi y_f\frak d$ is identifed with a fractional ideal of $\cO_F$ and $a(\frak m,\mathbf f)\in\bC$ is zero except if $\frak m$ is integral. 
We moreover say that such an $\mathbf f $ is normalized if $a(\cO_F,\mathbf f)=1$. Finally we say that $\mathbf f$ is primitive and normalized if we have the relation
$$L(\pi(\mathbf f),s- n)= \sum_{\frak m\subset \cO_F}a(\frak m,\mathbf f)N(\frak m)^{-s} $$
where $L(\pi(\mathbf f),s)$ stands for the standard $L$-function of $\pi(\mathbf f)$. Conversely, for any irreducible cuspidal representation $\pi=\pi_\infty\otimes \pi_f$ of $\GL_2(\bA_F)$ 
whose archimedean component $\pi_\infty$ is a principal series with central character $z\mapsto |z|_\C^{-n}$ with Casimir operators acting with the eigenvalue $n+n^2/2$ and with $\pi_f^{U^F_0(N)}$ of 
dimension $1$, there exists a unique normalized eigenform $\mathbf f$ such that $\pi=\pi(\mathbf f)$.

Let $\cO[X,Y]_n$ be the $\cO$-module of homogeneous polynomials in $X,Y$ of degree $n$
 with action of $g\in \SL_2(\cO_F)$ by 
where $g\cdot P(X,Y)=P((X,Y){}^tg^{-1})$; let $V_{k,k}(\cO)= \cO[X_1,Y_1]_n\otimes \cO[X_c,Y_c]_n$ 
with action of $g\in \SL_2(\cO_F)$ by $g\otimes g^c$.

Consider the (adelic) Bianchi threefold of level $N$:
$$X_F= \GL_2(\cO_F)\backslash\GL_2(\A_F)/U^F_0(N)\cdot\R^{\times}U_2(\C).$$
Let $\cV^F_{k,k}(\cO)$ be the local system on the Bianchi threefold $X_F$ associated to $V_{k,k}(\cO)$ and $\cV^F_{k,k}(\C):=\cV^F_{k,k}(\cO)\otimes_\cO\C$ its base change to $\C$
.
Let $\H^i_!$ be the image of the canonical map $\H^i_c\to\H^i$; Recall that for $i=1,2$, there are Hecke-linear Harder-Matsushima-Shimura isomorphisms 
$$\omega^i\colon S^F_{k,k}(N;\C)\cong \H^i_!(X_F,\cV^F_{k,k}(\C)) $$
given below (see \cite{Ur95} or \cite{H94}). 
The differential forms $\omega^i(\mathbf f)$ are defined by the formula
\begin{eqnarray*}\omega^1(\mathbf f)(\begin{pmatrix}y&x\\0&1\end{pmatrix}):=\sum_{0\leq j_1,j_c\leq n}
(-1)^{n-j_c}\left({n\atop j}\right) \left({n\atop j_c}\right) X_1^{n-j_1}Y_1^{j_1}\otimes X_c^{n-j_c}Y_c^{j_c} \times\\
y^{-1}\left(\mathbf f_{n+j_1-j_c}(g) dx - \frac{1}{2}\mathbf f_{n+j_1-j_c+1}(g)dy- \mathbf f_{n+j_1-j_c+2}(g) d\bar x\right)
\end{eqnarray*}
and
\begin{eqnarray*}\omega^2(\mathbf f)(\begin{pmatrix}y&x\\0&1\end{pmatrix}):=\sum_{0\leq j_1,j_c\leq n}
(-1)^{n-j_c}\left({n\atop j}\right) \left({n\atop j_c}\right) X_1^{n-j_1}Y_1^{j_1}\otimes X_c^{n-j_c}Y_c^{j_c} \times\\
y^{-2}\left(\mathbf f_{n+j_1-j_c}(g) dy\wedge dx-\mathbf f_{n+j_1-j_c+1}(g)dx\wedge d\bar x+ \mathbf f_{n+j_1-j_c+2}(g) dy\wedge d\bar x\right)
\end{eqnarray*}

\medskip
We now define periods attached to normalized primitive forms.
Let $\mathbf f$ be such a form and $\cO_{\mathbf f}$  the ring of integers containing 
$\cO_F$ and the Fourier coefficients of $\mathbf f$.
Then the $\cO_{\mathbf f}$-module
$\C.\omega^i(\mathbf f)\cap \H^i_!(X_F,\cV_{k,k}(\cO))\otimes \cO_{\mathbf f}/(tors)$ is projective of rank $1$.
Thus, for any localization $\cO_{(\wp)}$ of $\cO_{\mathbf f}$ at a finite prime $\wp\subset\cO_{\mathbf f}$, there exists  a basis
$\delta^i_{\mathbf f}$ over  $\cO_{(\wp)}$ of the localization of this module at $\wp$.  
We define therefore the periods $u^i(\mathbf f, \cO_{(\wp)})\in\C^\times/\cO_{\wp}^\times $  for $i\in\{1,2\}$, by the formula
\begin{equation}\label{periods}
 \omega^i(\mathbf f)=u_i(\mathbf f, \cO_{(\wp)})\cdot \delta^i_{\mathbf f}.
\end{equation}
When $\wp$ will be obvious from the context, we will just write $u^i(\mathbf f)$ for short.

\medskip
Let $h^F_{k,k}(N,\cO)$ be the Hecke algebra generated over $\cO$ by the Hecke operators
acting on the cohomology $\H_!^{\cdot}(X_F,\cV^F_{k,k}(\cO))$. Let $\lb_\ff$ be the character
$h^F_{k,k}(N,\cO)\rightarrow\cO$ associated to the Hecke eigensystem attached to the new form $\ff$
and let $\m:=\Ker (\lb_\ff\pmod\varpi)$ the corresponding maximal ideal. We then denote by $\T_{\m_\ff}$ the localization of 
$h^F_k(N,\cO)$ at $\m$.

\begin{lem} Assume $p$ is prime to $ND$ and that $\m_\ff$ is not Eisenstein if $p\leq k-2$. Then
$\H^1(X_F,\cV^F_{k,k}(\cO))_{\m_\ff}$ is a torsion free $\cO$-module. 
\end{lem}
\begin{proof}
 It  follows
from  the cohomology long exact sequence associated to the multiplication by $\varpi$ that $\H^0(X_F,V^F_{k,k}(\F))_{\m_\ff}$ surjects on the $\varpi$-torsion
of $\H^1(X_F,\cV^F_{k,k}(\cO))_{\m_\ff}$. It is equal to zero if $2<k<p+2$ since in that case  $V^F_{k,k}(\F)$ is irreducible as $U^F_0(N)$-module, or if ${\m_\ff}$ is not Eisenstein since the action of the Hecke algebra on the invariants is clearly Eisenstein. When $k=2$, $\H^0(X_F,\cO)_{\m_\ff}$ clearly surjects onto $\H^0(X_F,\F)_{\m_\ff}$, so the same long exact sequence implies the claim.
\end{proof}

To have a better understanding on the cohomology in degree $2$ which may have torsion in general, one needs to use the  existence and properties of Galois 
representations attached to Hecke eigensystems and the modularity lifting theorem techniques. 
Recall that it follows from the work of M. Harris, K.-W. Lan, R. Taylor and J. Thorne \cite{HLTT} , that there exists a Galois representation
$\rho_\ff \colon \Gamma_F\to \GL_2(\cO)$ associated to $\ff$. This results was proved also independently by  P. Scholze \cite{Scholze} -- who included the 
case of certain torsion classes -- whose work was further refined by J. Newton and J. Thorne in  \cite{NT16} and later by the work \cite{CGHJMRS}.  
There is therefore a Galois representation
$$\rho_{\T_{\m_\ff}}\colon \Gamma_F\to GL_2(\T_{\m_\ff})$$
unramified at finite places not dividing $Np$, and such that for all $v$ not dividing $Np$, $$tr\;\rho_{\T_{\m_\ff}}(\Frob_v)=T_v$$
where $T_v$  is the classical spherical Hecke operator associated to the double class of an element of $GL(F_v)$ with $v$-adic valuation $1$. If $\ell|N$ and is split in $F$, the for $v|\ell$, this representation satisfiees the expected local property at $v$. We refer to \cite{ACC+} for a precise statement.
We will however need to assume that this Galois representation satisfies some local properties at places dividing $p$. Recall that if $v|p$, one says that a representation of a decomposition subgroup at $v$ is Fontaine-Laffaille if 
it is the image by the Fontaine-Lafaille functor of a filtered $\Phi$-module (see \cite{FL} for more details).
We will consider the following condition when $p>k$:

\medskip
\item[(\bf FL)] For any  level $N$ prime to $p$,  the restriction of $\rho_{\T_{\m_\ff}}$ to any decomposition subgroup at $v|p$ is Fontaine-Laffaille.

\begin{rem}
It is conjectured that this property holds and it was proved in many cases  in the works \cite{ACC+, CGHJMRS} but the case of $F$ quadratic imaginary is left out because of some technical conditions. 
It is however hoped that it should be settled in the near future.
\end{rem}
We will say that $(Min_F)$ holds if the same condition as the one explained in section 2 before Theorem \ref{classical} does after we have replaced $\rho_f|_{\Gamma_F}$ by $\rho_{\ff}$. The following theorem gives the structure of the cohomology in degree $2$. It is follows from the combination of many deep works of several mathematicians.

\begin{thm}\label{CG} We suppose that $p>k$, $N$ is prime to $p$ and that the primes dividing $Np$ splits in $F$. We also assume that {\bf (FL)} is satisfied.
If $\m_\ff$ is non-Eisenstein and that $(Min_F)$ holds, then we have:
\begin{itemize}
\item[(i)] There is a presentation of $ \T_{\m_\ff}$ of the form
$$\T_{\m_\ff} \cong\cO[x_1,\dots,x_g]/(f_0,\dots,f_g)$$
\item[(ii)] $\H^2(X_F,\cV^F_{k,k}(\cO))_{\m_\ff}$ is free of rank one over $\T_{\m_\ff}$.
\end{itemize}
\end{thm}
\begin{proof}
This follows from the work of F. Calegari and D. Geraghty in  \cite{CaGe18} under the condition that some conjectures 
about the existence and properties of Galois representations
attached to torsion classes in $\H^2(X_F,\cV^F_{k,k}(\cO))_\m$ (see also the work of D. Hansen \cite{Han12}).
These conjectures are now proved (except {\bf (FL)} that we assume) in our special case thanks to the works \cite[Thm 3.1.1]{ACC+} and \cite[Thm 6.1.4] {CGHJMRS}.
The result of  \cite{CaGe18} is in terms of the homology $\H_1(X_F,\cV^F_{k,k}(\cO))_\m$ but this is an easy exercise to see that there is a canonical isomorphism
$\H^2(X_F,\cV^F_{k,k}(\cO))_{\m_\ff}\cong \H_1(X_F,\cV^F_{k,k}(\cO))_{\m_\ff}$ when $\m_\ff$ is not Eisenstein.
\end{proof}

Let us also record the following lemma that will be useful in the next section.
\begin{lem}\label{duality-Urban}
 Let $W^F_{N}$ be the Atkin-Lehner involution on $X_F$ and assume that $p>k-2$, the twisted cup-product $[x,y]=x\cup W^F_{N}y$ induces a  
perfect pairing 
$$[-,-]\colon \H_!^1(X_F,\cV^F_{k,k}(\cO))/(tors)\times \H_!^2(X_F,\cV^F_{k,k}(\cO))/(tors)\to \cO$$
which is Hecke-bilinear.
\end{lem}
\begin{proof} This follows from the Poincar\'e duality theorem \cite[Theorem 2.5.1]{Ur95},
\end{proof}

\subsection{Base change and congruence ideals}\label{sectHeckecong}
We consider now a cuspidal newform $f\in S_k(\Gamma_0(N))$ of weight $k\geq 2$ like in section
\ref{sectManinCongruence} and a prime $p$ satisfying the conditions of that section from which we will use the notations.

Let now $\ff=f_F\in S^F_{k,k}(N;\cO)$  be its normalized base change to $F$. This is the normalized eigenform 
attached to the quadratic base change $\pi(f_F)$ to 
$\GL_2(\A_F)$ of the cuspidal representation attached to $f$. We will denote by  $\m_F$ the maximal ideal of $h_{k,k}(N;\cO)$ associated to $f_F$ and
$\T_F$ the corresponding localization at $\m_F$. For $i=1,2$, we also write for short 
$\H^i_F:=\H^i(X_F,\cV_{k,k}(\cO))_{\m_F}$. We write $\T_F^{\f}$ for the maximal torsion free quotient of $\T_\F$.
The system of Hecke eigenvalues of $f_F$ defines a character of the Hecke algebra $\T_F \to \cO$ which factors through
$\T^{\ff}_F$ into an homomorphism $\lambda_F$
$$\lambda_F\colon \T_F^{\f}\to\T\to\cO$$
where $\theta\colon \T_F^{\f}\to\T$ is the base change homomorphism. 
%
%
%

We consider the congruence module of $f_F$ as in Section \ref{C0}
$$C_0(\lambda_F)=C_0(f_F)=\cO/\eta_{f_F}$$
and $\eta_{f_F}$ the corresponding Hecke congruence ideal.
Similarly, as in Section \ref{basechange}, we have the base change congruence modules and ideals
$$C_0^{\sharp}(f):=C_0^{\lb_F,\sharp}(\T_F)=\cO/\eta_f^\sharp $$
which control congruences between $f_F$ and Bianchi cusp forms which are not base change from $\Q$.

Finally, for $i\in\{ 1,2\}$, we define the congruences modules and ideals $\eta^i_f$, $\eta_f^{i,\sharp}$ attached to the Hecke modules $\H^i_F$:
$$C_0^{\lb_F}(\H^i_F)=\cO/\eta^i_{f_F}\qquad\hbox{ and }\qquad C_0^{\lb_F,\sharp}(\H^i_F)=\cO/\eta^{i,\sharp}_{f}$$
We have the following lemma
\begin{lem} \label{cong-rel} With the notation above, we have the following equalities or inclusion of ideals
\begin{itemize}
\item[(i)] $\eta^1_{f_F}=\eta^2_{f_F}\supset \eta_{f_F}$,
\item[(ii)] For $i\in\{ 1,2\}$,  we have $\eta^i_{f_F}\supset \eta_f\cdot \eta_f^{i,\sharp}$,
\item[(iii)] Assume (FL) and that the primes dividing $N$ are split in $F$. If $\m_F$ is non Eisenstein and minimal and $p>k$, then 
$$\eta^1_{f_F}=\eta^2_{f_F} =\eta_{f_F}\qquad\hbox{ and }\qquad\eta^\sharp_f=\eta_f^{2,\sharp}.$$
\item[(iv)] If $p>k-2$ and $p$ does not divide $N\phi_F(N)$, then for $i\in\{ 1,2\}$ we have 
$$\eta_{f_F}^{i}\sim{ L(\Ad(f_F),1)\over \pi^{2k+1} u_1(f_F) u_2(f_F)}
$$
\end{itemize}
\end{lem}
 \begin{proof} The equality in (i) follows from the duality of Lemma \ref{duality-Urban} and Proposition \ref{Pontryaginprop}. 
The inclusion follows from \eqref{div-C0}.
 The point (ii) follows from Lemma \ref{BC-congruenceN} for $M=\H^i_F$ and that $\eta_{f}(\H^i(F)_{\T})\supset\eta_f$ thanks to \eqref{div-C0}.
The point (iii) is an immediate consequence of the point (ii) of
 Theorem \ref{CG}. The point (iv) is a reformulation of the main result of \cite{Ur95}(see also \cite{H99}) and follows from the last point of 
 Proposition \ref{Pontryaginprop} applied to $M^i=\H^i(F)$ and $\delta^i_{f_F}$ and Lemma \ref{duality-Urban}
 together with the computation of the Petersson inner product of $f_F$ in terms of the value of Adjoint L-function at $s=1$.
 \end{proof}

\subsection{The integral period relation and base change congruence ideal}


We keep the notations and hypothesis for $f$ as in the previous section. We prove a first $p$-adic divisibility of periods:

\begin{pro}\label{properiodiv} Assume that $p>k$ and $p$ does not divide $DN\phi_F(N)$. 
We also suppose that (CV), (Irr)  hold and furthermore  that the ideal of prime factors of $N$ of the level of $f$
which remain inert in $F$ is odd. Then, the quotient 
$${\Omega^+_f\Omega^-_f\over u_1(f_F)}\in\overline\Z_p.$$
\end{pro}

\begin{proof} The proof follows the same strategy as in the real quadratic case.
Since $f$ has trivial central character and since the imaginary quadratic field is such that the ideal of prime factors which
remain inert in $F$ is odd, we know, by \cite{C02} and \cite{Va02} 
(for the weight $k=2$) and Chida-Hsieh \cite{CH16} (for $k$ such that $k-2<p$),  that
for a prime $\ell$ prime to $Np$, for almost all anticyclotomic Hecke characters of $F$ of $\ell$-power conductors
we have
\begin{equation} \label{deno} {L(f_F\otimes\psi,k/2)\over \pi^k \Omega^+_f\Omega^-_f}  ={L(f\otimes\theta(\psi),k/2)\over \pi^k\Omega^+_f\Omega^-_f}\in\overline{\Z}_p^\times
\end{equation}
But on the other hand, by integrating the $1$-form $\delta^1_{f_F}={\omega^1(f_F)\over u_1(f_F,\cO_\wp)}$ against the modular symbols  of
the Bianchi hyperbolic threefold $X_F$, we see that
\begin{equation}\label{num} {L(f_F\otimes\psi,k/2)\over \pi^k u_1(f_F)}\in\overline\Z_p.
\end{equation}
 Note that such an integration makes sense because after localization at a non-Eisenstein maximal ideal, the homology relative to the boundary where  the modular symbols belong is the same as the homology. The above formula follows from the same computations as in the proof of Theorem 8.1 of \cite{H94}. In loc. cit. the number field $A$ has to be replaced by our discrete 
valuation ring $O_{(\wp)}$ and  the basis $\xi_\epsilon$ of the Hecke eigenspace in $\H^1_!(X_F,\mathcal V_{k,k}(A))[\lb_{f_F}]$ by our basis
 $\delta^1_{f_F}$ of the $\H^1_!(X_F,\mathcal V_{k,k}(O_{(\wp)})) [\lb_{f_F}] $. Note that the fudge factors are units in $\overline \Z_p$ by our assumptions on $p$.
Our  claim now follows from  dividing \eqref{num} by \eqref{deno}.
\end{proof}

Like in the real quadratic case, the next proposition will allow us to deal with the reverse divisibility. It makes use of the period linear form of Hida \cite{H99} which is defined 
in the  imaginary quadratic case as well.

\begin{pro} \label{propHidadual} Assume that $p>k-2$ and is prime to $ND\phi_F(N)$, then
we have
$$ {L(\Ad(f),\alpha,1)\over \pi^k u_2(f_F)}\in  \eta_{f_F}^{1,\sharp}.$$
\end{pro}
 
 \begin{proof} 
The proof is totally similar to the one of Proposition \ref{Hida-linearform}.
Let $X\subset X_F$ be the modular curve view as a $2$-cycle of the Bianchi orbifold via the natural embedding of $\GL(2)_{/\Q}$ into $\GL(2)_{/F}$.
Then one can define a linear form
$$ \H^2_c(X_F,\cV^F_{k,k}(\cO))\to \H^2_c(X,\cV_{k,k}(\cO)) \to \H^2_c(X,\cO))\cong  \cO$$
where again the first map is induced by a character twist followed by the restriction $\omega\mapsto \omega_{X}$ and the second map by the canonical projection of $\GL_2(\Zp)$-representations
$V_k(\cO) \otimes V_k(\cO)\to\cO$ induced by the Clebsch-Gordan decomposition. After localization at $\m_F$, this defines a liner form $Ev\colon  \H^2_!(X_F,\cV^F_{k,k}(\cO))_{\m_F}\to \cO$. 
 Hida \cite[Prop. 3.1]{H99} proved that  if ${\ff}$ is 
not the twist of a base change, $Ev(\omega^2({\bf f}))=0$ while if ${\bf f}=f_F$ is a base change,
$$Ev(\omega^2({\bf f})) \sim{ L(\Ad(f),\alpha,1)\over\pi^k}$$
Then our claim follows from Proposition \ref{linear-BC} with $\Phi$ the restriction of $\Ev$ to the $\cO$-torsion-deprived localized second cohomology module $ \H^2_F/(tors)$
and $\delta=\delta^2_{f_F}\in \H^2_F/(tors)$
since $\H^1_F$ is the $\cO$-dual of $\H^2_F$ by Lemma \ref{duality-Urban}.
 \end{proof}

We can now prove the main result of this section.

\begin{thm} \label{ThNBC} Assume that that $(N,D)=1$, that $p$ does not divide $DN\phi_F(N)$ and $p>k$.
Assume also (CV), (Irr). Then the following hold:
\begin{itemize}
\item[(i)] We have the integral period relation  $u_1(f_F)\sim \Omega_f^+\Omega_f^-$.
\item[(ii)] We have the equality (up to $p$-adic unit)
$$\eta_{f_F}^{1, \sharp}\sim {L(\Ad(f)\otimes\alpha,1)\over \pi^k u_2(f_F)}$$

\item[(iii)] We have the following ideal factorization  in $\cO$
$$\eta_{f_F}^{1}= \eta_f\cdot \eta_{f}^{1,\sharp}
$$
\end{itemize}
\end{thm}  

\begin{proof} 
Recall that from Theorem \ref{classical}, we have
\begin{equation}
\label{classical-eq}
\eta_f\sim {L(\Ad(f),1)\over \pi^{k+1}\Omega^+_f\Omega^-_f}.
\end{equation}
On the other hand, from the point (iv) of Lemma \ref{cong-rel}, we have
\medskip
$$\eta_{f_F}^1\sim {L(\Ad(f_F),1)\over \pi^{2k+1} u_1(f_F)u_2(f_F)}={L(\Ad(f,1)\over \pi^{k+1}u_1(f_F)}.\cdot 
{L(\Ad(f)\otimes\alpha,1)\over \pi^k u_2(f_F) }= a\cdot {L(\Ad(f,1)\over \pi^{k+1}\Omega^+_f\Omega^-_f} \cdot {L(\Ad(f)\otimes\alpha,1)\over \pi^k u_2(f_F)}$$
\medskip
where $a={\Omega^+_f\Omega^-_f\over u_1(f_F)}\in\cO$ thanks to Proposition \ref{properiodiv}. From \eqref{classical-eq}, we therefore deduce that
$$\eta^{1}_{f_F}\sim a\cdot \eta_f\cdot {L(\Ad(f)\otimes\alpha,1)\over \pi^k u_2(f_F)}.$$
On the other hand, from the point (ii) of Lemma \ref{cong-rel}, we have 
$$\eta^{1}_{f_F}\supset \eta_{f}\cdot\eta^{1,\sharp}_{f_F}.$$
After simplifying by $\eta_f$, we therefore deduce that
$$ {L(\Ad(f)\otimes\alpha,1)\over \pi^k u_2(f_F)}\cdot\cO \supset a\cdot   {L(\Ad(f)\otimes\alpha,1)\over  \pi^k u_2(f_F)} \cdot\cO
\supset  \eta^{1,\sharp}_{f_F}$$
Now applying Proposition \ref{propHidadual}, we deduce the point (ii) and that all the inclusion above are equalities  which yields the point  (iii) and that
$a$ is a $p$-adic unit which is the point (i) of the Theorem.
\end{proof}

\begin{rem}
The fact that such a proof could work may seem a bit surprising. In fact, one of the reason behind the importance of $Ev$ in our proof  is that $(H^1_F)_{\T}$ can be proved to be free of rank one over $\T=\T_\Q$ under the assumption of Theorem \ref{CG}, and  that it has a canonical generator given by $Ev$. Note first that by Lemma \ref{duality-Urban}, $(H^1_F)_{\T}$ can be identified as $Hom_\cO(H^2_F/\Ker(Ev),\cO)$. Since under the assumption of Theorem \ref{CG}, $H^2_F$ is free of rank one over $\T_F$, it follows that $(H^2_F)^{\T}=H^2_F/\Ker(Ev)$ modulo $\cO$-torsion is free of rank one over $\T$ because  $Ker(Ev)=(H^2_F)_{\T_F^\sharp}$. Since $\T$ is Gorenstein, this implies that its $\cO$-dual $(H^1_F)_{\T}$  is also free of rank one over $\T$ as claimed. Therefore  we have 
$\eta_{\lb_f}((H^1_F)_{\T})=\eta_f$ and  this implies the point (iii) of the theorem by Lemma \ref{BC-congruenceN}. Any generator of the $\cO $-dual of 
$H^2_F/(H^2_F)_{\T_F^\sharp}$ 
would now produce a generator of $\eta^{1,\sharp}$. The point (ii) of our theorem implies actually by Hida's formula that $Ev$ is such a generator. Indeed, let $\Phi$ be a generator, then by Proposition \ref{linear-BC}, we would have
$\Phi(\delta_{f_F}^2)\in \eta_f^{\sharp,1}$. On the other hand, we would have also $Ev=U.\Phi$ for some Hecke operator $U\in\T$ and therefore $Ev(\delta_{f_F}^2)=\lb_f(U)\Phi(\delta_{f_F}^2) \in \lb_f(U)\cdot \eta_f^{\sharp,1}$. Now, by Hida's formula and (ii) of the previous theorem, we now that $Ev(\delta_{f_F}^2)$ is a generator of $\eta_f^{\sharp,1}$. This implies that $\lb_f(U)\in\cO^\times$ and therefore $U\in \T^\times$. In other words, $Ev$ is a generator of $(H^1_F)_{\T}$ as claimed.

Note that the proof of Theorem \ref{ThNBC} does not use Theorem \ref{CG} and its stronger hypothesis. It could be interesting to know that
the equality $\eta_{\lb_f}((H^1_F)_{\T})=\eta_f$ that follows from (iii)  implies that $(H^1_F)_{\T}$ is free over $\T$ but the authors ignore how to prove such a fact.
\end{rem}
\begin{rem} We don't know whether the following equality of ideals
$$\eta_{f_F}^{2}= \eta_f\cdot \eta_{f_F}^{2,\sharp}
$$
holds or not. Or, in other words, if $\eta_{f_F}^{i,\sharp}$ is independent of $i$. 
\end{rem}

We have the following immediate corollary
\begin{cor} We keep the same hypothesis as the previous Theorem, then we have the divisibility
$$\eta_f^\sharp \subset {L(\Ad(f)\otimes\alpha,1)\over \pi^k u_2(f_F)}\cdot \cO$$
In particular, if $\wp$ divides $ {L(\Ad(f)\otimes\alpha,1)\over \pi^k u_2(f_F)}$, there exists a Bianchi eigenform which is not a base change from $\Q$ which is congruence to $f_F$
modulo $\wp$.
\end{cor}
\begin{proof}
This follows from the point (ii) of the previous Theorem and the natural inclusion $\eta_f^{1,\sharp}\supset\eta_f^\sharp$.
\end{proof}


 \subsection{Relation to the Bloch-Kato conjecture}\label{subsectBKimquad}
 
Our goal in this section is to establish, under the assumptions of Theorem \ref{CG},
an expression of the order of the Bloch-Kato Selmer group for $\Ad \rho_f\otimes\alpha$ in terms of the non critical value
$L(\Ad f\otimes\alpha,1)$, with suitable integral periods (see Theorem \ref{ThOPT} below).

We first review the Bloch-Kato conjecture in this context. This naturally leads us to a  formulation of
an integral version of a special case of a conjecture of Prasanna-Venkatesh.
 
\medskip
Let $M_f$ the Grothendieck motive over $\Q$ associated to our eigenform $f$. Let $\Ad(M_f):=Sym^2(M_f)(k-1)$ 
be the Adjoint motive defined over $\Q$ associated to the modular form $f$. 
It is of rank $3$ over the coefficient field $K_f\subset\bar\bQ$. 
We consider its twist $W_{f,\alpha}:=\Ad(M_f)\otimes\alpha$ by our odd quadratic Dirichlet character $\alpha$.
For any motive $X$, we denote by $\H_?(X)$ the $?$-realization functors with $?=B,dR$ and $et$ for respectively 
the Betti, de Rham and \'etale $p$-adic realizations. Recall that we have the following comparison isomorphisms
between these realizations.

\medskip
\begin{itemize}
\item[(i)] $c_\infty\colon \H_B(W_{f,\alpha})\otimes_{K_f}\C\cong \H_{dR}(W_{f,\alpha})\otimes_{K_f}\C$
\item[(ii)] $c_p\colon \H_B(W_{f,\alpha})\otimes_{K_f} K \cong \H_{et}(W_{f,\alpha})=\Ad(\rho_f)\otimes K(\alpha)$
\item[(iii)] $c_{dR}\colon D_{dR}( \H_{et}(W_{f,\alpha}))\cong \H_{dR}(W_{f,\alpha})$.
\end{itemize}

\medskip

Since $L(W_{f,\alpha},1)\neq 0$, we see that the $L$-function
$L(W_{f,\alpha},s)$ vanishes at $s=0$ with order one. It follows from the functional equation and because its Gamma factor $\Gamma(W_{f,\alpha},s)=\Gamma_\C(s+k-1)\Gamma_{\R}(s)$
has a simple pole at $s=0$.  Therefore, 
a conjecture of Beilinson predicts that the motivic cohomology group
$$\H^1_{mot}(W_{f,\alpha},K_f(1))=Ext^1_{\mathcal {MM}}(W_{f,\alpha},K_f(1))$$ 
is of rank 
$1$ over $K_f$. Here $\mathcal{MM}$ is the conjectural category of mixed motives.
Part of the conjectural framework is the existence of  archimedean  and $p$-adic regulator maps:
$$
\xymatrix{ \H^1_{mot}(W_{f,\alpha},K_f(1))\ar^{r_p}[r]\ar[d]_{r_\infty} &  \H^1_f(\Q,\H_{et}(W_{f,\alpha})(1)) \\
  \H^1_{\mathcal D}(W_\R, \R(1))&
}
$$
where  $\H^1_f(\Q,\H_{et}(W_{f,\alpha})(1))=\H^1_f(\Q,\Ad(\rho_f)\otimes \alpha(1))$ is the Bloch-Kato Selmer group attached 
to the Galois representation $\Ad(\rho_f)\otimes\alpha(1)$ and 
$\H^1_{\mathcal D}$ stands for Deligne cohomology where $W_\R$ denotes the real Hodge structure of $\H_{B}(W_{f,\alpha}(1))\otimes \R$. 
By a standard computation,
the Deligne cohomology is equal to the cokernel of Deligne's period map
\begin{eqnarray*}
\H_B(W_{f,\alpha}(1))^+_\R\stackrel{c_\infty^+}{\longrightarrow} \H_{dR}(W_{f,\alpha}(1))_\R /F^- \H_{dR}(W_{f,\alpha}(1))_\R\stackrel{\pi_\infty}{\longrightarrow} \H^1_{\mathcal D}(W_\R, \R(1))
\end{eqnarray*}
where $c_\infty^+$ is injective because the motivic weight of $W_{f,\alpha}(1)$ is negative. 
Note that with the original Deligne's Conjecture paper notations, 
$F^-\H_{dR}(W_{f,\alpha}(1))_\R=F^0 \H_{dR}(W_{f,\alpha}(1))_\R$. Therefore the right-hand side is dimension $2$ 
for any character $\alpha$ 
while the left hand side is dimension $2$ or $1$ according to whether $\alpha$ is even or odd. 
Therefore our motive is not critical when $\alpha$ is attached to an imaginary quadratic extension. 
One sees immediately 
that the Deligne cohomology has a $K_f$-rational structure $\pi_\infty(\H_{dR}(W_{f,\alpha}(1)))$ . 

If follows from the work of Beilinson that $r_\infty\neq 0$; in fact, the space $\H^1_{mot}(\Ad(M_f)\otimes\alpha,K_f(1))$ contains 
a subspace $\Phi_{f,\alpha}$ of dimension 1 over $K_f$ generated by the so-called Beilinson-Flach elements such that
$r_\infty(\Phi_{f,\alpha})\neq 0$. We can define a regulator $R_{f,\alpha}\in \R^\times/K_f^\times$
as the period of $r_\infty(\Phi_{f,\alpha})$ with respect to the $K_f$-rational structure $\pi_\infty(\H_{dR}(W_{f,\alpha}(1)))$:
$$\frac{r_\infty(\Phi_{f,\alpha})}{R_{f,\alpha}}\in \pi_\infty (\H_{dR}(W_{f,\alpha}(1))).$$

We want to define a $p$-adic integral version of this regulator.
Let us assume $p>2k-1$ and $p$ prime to $N$. We consider the two conditions:

\medskip
\item[${\mathbf (Sel)}$] The Bloch-Kato Selmer group of $H^1_f(\Q,Ad(\rho_f)(1)\otimes\alpha)$ is of rank $1$ over $K$.

\medskip
\item[${\mathbf (Reg)_p}$]    \hskip 20pt  $ r_p(\Phi_{f,\alpha})\neq 0$.

\medskip
\begin{rem}
1) When $f$ is ordinary at $p$, it follows from the work of \cite{LZ16} that ${\mathbf (Reg)_p}$ is equivalent to the non-vanishing 
$L_p(Ad(f)\otimes\alpha,1)\neq 0$. It is very likely that by using the method of \cite{Ur06}, one could prove that ${\mathbf (Sel)}$ implies ${\mathbf (Reg)_p}$. We hope to come back to this in a near future.

2) Without assuming ordinarity, but assuming that $p$ prime to $N$ and $p>2k-1$, if follows from Lemma \ref{Sha-Selmer} below
that ${\mathbf (Sel)}$ holds
under the assumptions of Theorem \ref{CG}. However, the validity of ${\mathbf (Reg)_p}$ remains unknown.
\end{rem}

Under these assumptions, we can define an "integral regulator" $R_{f,\alpha}\in\R^\times/\cO_{(\wp)}^\times$ as follows. 
Let $L_f\subset \H_B(M_f)$ be the $\cO_{(\wp)}$-lattice coming from the integral Betti cohomology of the modular curve and $T_f$ the corresponding
$\cO$-lattice  in $\H_{et}(M_f)=\rho_f$.
Applying $\Symm^2\otimes \det^{-1}\otimes\alpha$, we get a lattice $\Lambda_{f,\alpha}$ in $\H_B(W_{f,\alpha})$ and a stable 
$\cO$-latice $T_{f,\alpha}$ in $\H_{et}(W_{f,\alpha})$. 
The lattices $\Lambda_{f,\alpha}$ and $T_{f,\alpha}$ correspond by $c_p$.

Since $p>2k-1$, the Hodge-Tate weights of $W_{f,\alpha}$ are in the Fontaine-Lafaille range and
 $$D_{crys}(T_{f,\alpha})=\Ad (D_{crys}(T_f))\otimes D_{crys}(\alpha)$$
 This lattice defines  
an integral $\cO_{(\wp)}$- lattice $T_{dR,f,\alpha}$ of $\H_{dR}(W_{f,\alpha})$. 
This yields an  $\cO_{(\wp)}$ -integral structure on $\Coker(c_\infty^+)=
 \H^1_{\mathcal D}(\Ad(M_\R)\otimes\alpha, \R(1))$, hence a Haar measure (well defined up to $\cO_{\p)}^\times$) giving volume $1$ to this lattice.
 Note that by \cite[Theorem 4.1]{BK90}, since the Hodge-Tate weights are in the Fontaine-Laffaille range, 
this Haar measure gives the correct Euler factor of $L(W_{f,\alpha},0)$ at $p$.
 
Let $A(\Q)$ be a projective $O_{K_f}$-submodule of $\Phi_{f,\alpha}$ of rank 1 such that its localization at $\wp$ is isomorphic to the inverse image of $ \H^1_f(\Q,T_{f,\alpha})$ via $r_p$ in $\Phi_{f,\alpha}$.
The archimedean regulator of $W_{f,\alpha}(1)$ is then defined as the co-volume of the image of
 $A(\Q)$ by $r_\infty$ with respect to the Haar measure induced by the $\cO_{(\wp)}$ -integral structure defined above. It gives an element
 $$R_{f,\alpha} \in\R^\times/\cO_{(\wp)}^\times$$ 

We now review the definition of the Tate-Shafarevich group $\Sha^1(\Q,\Ad(M_f)(1)\otimes\alpha)$ following Bloch-Kato. It is defined as:
$$\Ker\left({\H^1(\Q,T_{f,\alpha}\otimes K/\cO)\over \Phi_\cO\otimes K/\cO}\to\bigoplus_v
{\H^1(\Q_v,T_{f,\alpha}\otimes K/\cO)\over \H^1_{\f}(\Q_v,T_{f,\alpha}\otimes  K/\cO)}\right)$$
where  $\Phi_\cO= \H^1_{\f}(\Q,\Ad(\rho_f)(1)\otimes\alpha) \cap \H^1(\Q,T_{f,\alpha})$.

The Bloch-Kato conjecture in this situation can be formulated as:
 
 \begin{conj} \label{BK-non-crit} We assume that $p>2k-1$, $(p,N)=1$, $(Sel)$ and $(\Min_F)$ hold. Then, the Tate-Shafarevitch group $\Sha^1(\Q,\W_{f,\alpha}(1))$ is finite and
 $$\Gamma(\Ad(f)\otimes\alpha,1) \frac{L(\Ad(f)\otimes\alpha,1)}{\Omega_f^+\Omega_f^-\cdot R_{f,\alpha}}\sim  \CH \Sha^1(\Q,\Ad(M_f)(1)\otimes\alpha)$$
 \end{conj}
 \proof We want to explain how the computation of the period using the Tamagawa measure defined as in \cite{BK90} yields the above formulation
 of their conjecture in our case. With the notations of \cite[Section 5]{BK90}, recall that the Tamagawa number attached to $W_{f,\alpha}$ is a product 
 $\prod_{p\leq\infty}\mu_p$ where $\mu_\infty$ is the measure of $A(\R)/A(\Q)$ with
 $$A(\R):=D_\R/(F^0 D_\R +W^+)$$
 with $D=\H_{dR}(W_{f,\alpha}(1))$,  $W=c_p^{-1}( \Ad(L_f)(1)\otimes\alpha))\cap \H_B(W_{f,\alpha}(1))$ and where we have identified $A(\Q)$ to its image under the regulator map $r_\infty$. We have the exact sequence
 $$0\rightarrow W^+_\R/W^+\rightarrow A(\R)/A(\Q)\rightarrow  \H^1_{\mathcal D}(\Ad(M_\R)\otimes\alpha, \R(1))/A(\Q)\rightarrow 0$$
 The volume of the right-hand side quotient is $R_{f,\alpha}$ by definition. It remains to compute the volume of $ W^+_\R/W^+$. 
Since $M_f^\vee\cong M_f(k-1)$, we have
 $$W \cong  Sym^2(L))(k)\otimes\alpha$$
 with $L\subset \H_B(M_f)\otimes \cO_{(\wp)}$ the inverse image of $L_f$ by the $p$-adic \'etale comparison map 
$\H_B(M_f)\otimes K\cong  \H_p(M_f)\cong L_f\otimes K$.
Since $\delta^+_f,\delta^-_f$ is an $\cO_\wp$-basis of $L$, we see from the above isomorphism that $\delta^+_f\otimes\delta^-_f$ is an $\cO_\wp$-basis
 of $W^+$. Since the measure on $W^+_\R$ is induced by the $\cO_\wp$-structure of the $D/F^0D$, we deduce the volume of $ M^+_\R/M^+$ is given by $(2i\pi)^k\Omega_f^+\Omega_f^-$.
 From \cite{BK90}, we deduce that 
 $$Tam(W_{f,\alpha}(1))=(2i\pi)^k\Omega_f^+\Omega_f^-\cdot R_{f,\alpha} L^N(\Ad(f)\otimes\alpha,1)^{-1}\prod_{\ell|N}\mu_\ell$$
 where $\mu_\ell$ are the local Tamagawa numbers at the places dividing the level $N$. They can be computed as follows as in \cite{DFG},
 $$
 \mu_\ell=Tam_\ell(\Ad(\rho_f)(1)\otimes\alpha)  =  Tam_\ell(\Ad(\rho_f)\otimes\alpha) = { \# \H^1_f(\Q_\ell,\Ad(\rho_f)\otimes\alpha) \over \# \H^0(\Q_\ell,\Ad(\rho_f)\otimes \alpha) } = 1
 $$ 
 by the minimality hypothesis $(\Min_F)$.
 Because $\Gamma(\Ad (f)\otimes\alpha,1)=2\cdot(2\pi)^{-k}$, we deduce that 
 the Tamagawa number conjecture asserting that $Tam(W_{f,\alpha}(1))=\#\Sha^1(\Q,W_{f,\alpha}(1))^{-1}$ is equivalent to our formulation.\qed
 
\medskip

\begin{lem}\label{Sha-Selmer}
Under the hypothesis of Theorem \ref{CG}, $\H^1_f(\Q,\Ad(\rho_f)\otimes\alpha(1))$ is of rank one over $K$, hence ${\mathbf (Sel)}$ holds. 
Moreover
$\Sha^1(\Q,\Ad(M_f)\otimes\alpha(1))$ and $\H^1_f(\Q,\Ad(\rho_f)\otimes K/\cO(\alpha))$ have same $\cO$-Fitting ideal.
\end{lem}
\begin{proof}
By the results of \cite{CaGe18}, we know that the noetherian ring $R$ is finite over $\cO$ (because $R=\T$).
Hence $\Omega_{R/\cO}\otimes_{\lb}\cO$ is finite and the $K$-vector space $\H^1_f(\Q,\Ad(\rho_f)\otimes\alpha)$ is $0$. 
Now the rest of the Lemma follows easily from Poitou-Tate duality considerations. 
Let us write $\rho_f\colon G_{F,S}\to\GL_2(\cO)$ for the integral representation associated to $f$ and
$W_n:=\Ad(\rho_f)\otimes \varpi^{-n}\cO/\cO(\alpha)$. Let $q=\#\F$. Since $\H^0(\Q,W_n)=\H^0(\Q,W_n(1)=\{0\}$, we have
by the Euler-Poincar\'e characteristic formula due to Greenberg-Wiles 
\begin{eqnarray*}
{\# \H^1_f(\Q,W_n)\over \# \H^1_f(\Q,W_n(1))}= {1\over \#\H^0(\R,W_n)} \cdot \prod_{\ell|Np} {   \# \H^1_f(\Q_\ell, W_n)  \over \# \H^0(\Q_\ell,W_n)      }
\end{eqnarray*}
We note that because of the minimality condition, we have $ \# \H^1_f(\Q_\ell, W_n)  = \# \H^0(\Q_\ell,W_n)$ for all $\ell | N$.
Now since $\H^0(\R,W_n)=(\cO/\varpi^n\cO)^2$ (because $\alpha$ is odd ) and  ${\#\H^1_f(\Q_p,W_n)\over \# \H^0(\Q_p,W_n)}=q^n$ 
by \cite[Corollary 2.35]{DDT95}, we get
\begin{eqnarray}\label{GWPT}
{\# \H^1_f(\Q,W_n(1))= q^n\cdot  \# \H^1_f(\Q,W_n)   }
\end{eqnarray}
We conclude using the exact sequence
$$0\rightarrow \Phi_\cO\otimes \varpi^{-n}\cO/\cO\rightarrow \H^1_f(\Q,W_n(1))\rightarrow \Sha^1(\Q,\Ad(M_f)(1)\otimes\alpha)[\varpi^n]\rightarrow 0$$
From  \eqref{GWPT}, we indeed deduce that
$\#\Sha^1(\Q,\Ad(M_f)(1)\otimes\alpha)[\varpi^n] = \# \H^1_f(\Q,W_n)$ for all $n>0$. Our claim follows now invoking the finiteness of $\H^1_f(\Q,\Ad(\rho_f)\otimes\alpha)$.
\end{proof}

Let us now relate our result to the Bloch-Kato conjecture. 
\begin{thm} \label{ThOPT} Assume that $p$, $D$ and $N$ are relatively prime and $p>k$, $(Irr_F)$ and $(\Min_F)$. Then, we have

$$\CH\,\Sel(\Q,\Ad(\rho_f\otimes\alpha)) \sim {L(\Ad(f)\otimes\alpha,1)\over \pi^{k} u_2(f_F)}\cdot \CH  \H_1(L_{T_F/\cO}\otimes_{\lb_{f_F}}\cO)$$
where $L_{T_F/\cO}$ stands for the cotangent complex of $T_F$ over $\cO$.
\end{thm}
\begin{proof}
Since $R_F\cong T_F$ by the work of Calegary-Geraghty, using Mazur's isomorphism, we know that 
$$\Sel(F,\Ad(\rho_f))^*\cong \Omega_{T_F/\cO}\otimes\cO.$$
Moreover by Theorem \ref{CG}, Lemma  \ref{cong-rel} and Theorem \ref{ThNBC}.(iii), we have
$$\eta_{f_F}=\eta^2_{f_F}=\eta^1_{f_F}=\eta_f\cdot\eta_f^{1,\sharp}$$
Therefore it follows from Proposition \ref{Wiles-defect} that
\begin{eqnarray*}
\CH \Sel(F,\Ad(\rho_f)& = & \eta_{f_F}\cdot\CH \H_1(L_{T_F/\cO}\otimes_{\lb_{f_F}}\cO)\\
&=&\eta_f\cdot\eta_f^{1,\sharp}\cdot \CH \H_1(L_{T_F/\cO}\otimes_{\lb_{f_F}}\cO)\\
\CH \Sel(\Q,\Ad(\rho_f))\cdot \CH \Sel(\Q,\Ad(\rho_f)\otimes\alpha)&=& \eta_f\cdot\eta_f^{1,\sharp}\cdot \CH \H_1(L_{T_F/\cO}\otimes_{\lb_{f_F}}\cO)\\
\end{eqnarray*}
Now we know that $\CH \Sel(\Q,\Ad(\rho_f))= \eta_f$ by Wiles since $R_\Q\cong \T$ is complete intersection. We therefore get
$$  \CH \Sel(\Q,\Ad(\rho_f)\otimes\alpha)=  \eta_f^{1,\sharp}\cdot \CH \\H_1(L_{T_F/\cO}\otimes_{\lb_{f_F}}\cO)$$
and we can conclude using Theorem \ref{ThNBC}.(ii).
\end{proof}
\begin{rem} Since in this situation, the local Hecke algebra $\T_F$ is not complete intersection, we introduced the Wiles defect
$$\delta_{\lb_{f_F}}=\CH_\cO \H_1(L_{T_F/\cO}\otimes_{\lb_{f_F}}\cO).$$
The correct "integral automorphic Bloch-Kato period" is therefore
$$\widetilde{u}_2(f_F)=u_2(f_F)\cdot\delta_{\lb_{f_F}}^{-1}.$$ 
\end{rem}
\medskip
The following Proposition is independent of our results on period relations. We record it with the hope it shed some light on
the structure of the Selmer group of $\Ad(\rho_f)\otimes\alpha$.
\begin{pro} We keep the notations and  hypothesis as the previous Theorem. Then, we
have an isomorphism
$$\Sel(\Q,\Ad(\rho_f)\otimes\alpha)^*\cong \H^2(F)_\sharp\otimes_{\lb_{f_F}}\cO$$
where $\H_2(F)_\sharp=\H^2(F)\otimes_{\T_F}\Ker \theta$ with $\theta\colon \T_F\rightarrow \T$ is the base change homomorphism.
\end{pro}
\begin{proof}
This follows from Lemma \ref{C1BCle} and the fact that $\H^2(F)$ is free over $\T_F$ by Theorem \ref{CG}.
\end{proof}
\medskip
In our situation, a special case of a conjecture of \cite{PV16} gives a canonical identification respecting
the rational structures
\begin{equation}\label{PV}
\H^1_{mot}(\Ad(M_f)\otimes \alpha,K_f(1))\otimes \C \cong \Hom(\H^1_{f_F},\H^2_{f_F})\otimes\C
\end{equation}
where $\H^i_{f_F}$ stands for the $f_F$- part of the degree $i$ singular cohomology of $X_F$ with coefficient in $K_f$. In other words,
$
{u_2(f_F)^{-1} u_1(f_F)\cdot R_{f,\alpha}}\in\overline K_f
$
If we take into account Theorem \ref{ThNBC}.(i), the following conjecture is therefore an integral version of \eqref{PV}.
\medskip
\begin{conj} \label{PV_int}Assume that $p>2k-1$, $(p,N)=1$, $(Reg)_p$ and that $(\Min_F)$ hold, then the following holds:
$$  \widetilde{u}_2(f_F)\sim  {\Omega_f^+\Omega_f^-\cdot R_{\f,\alpha}}.$$
\end{conj}

We deduce

\begin{cor} Assume the assumptions of Theorem \ref{ThOPT} hold. Then, Conjecture \ref{PV_int} and  Conjecture \ref{BK-non-crit}
are equivalent.
\end{cor}
\begin{proof}
This is immediate from Theorem \ref{ThOPT} and Lemma \ref{Sha-Selmer}.
\end{proof}

{}

J. Tilouine, LAGA, Institut Galil\'ee, Universit\'e de Paris 13, E. Urban, Dept of Mathematics, Columbia University
\end{document}